\documentclass[11pt]{amsart}  
\usepackage{geometry}                	
\usepackage{graphicx}		
\usepackage{amsmath}		
\usepackage{amssymb}
\usepackage{graphicx}					
\usepackage{amssymb}
\usepackage{amsmath}
\usepackage{amsfonts}
\usepackage{amsthm}
\usepackage{mathrsfs}
\usepackage{xcolor}
\usepackage{tikz-cd}
\usepackage{mathtools}
\usepackage{hyperref}

\theoremstyle{plain}
\newtheorem{theorem}{Theorem}[section]
\newtheorem{corollary}[theorem]{Corollary}
\newtheorem{lemma}[theorem]{Lemma}
\newtheorem{proposition}[theorem]{Proposition}
\newtheorem{conjecture}[theorem]{Conjecture}

\newtheorem{algorithm}[theorem]{Algorithm}

\theoremstyle{definition}
\newtheorem{definition}[theorem]{Definition}

\newtheorem{example}[theorem]{Example}

\theoremstyle{remark}

\newtheorem{remark}[theorem]{Remark}

\newcommand{\defeq}{\vcentcolon=}

\DeclareMathOperator{\red}{red}

\begin{document}
\title[$p$-adic Higher Green's functions for Stark-Heegner Cycles]{$p$-adic Higher Green's Functions for Stark-Heegner Cycles}
\author{Hazem Hassan}
\address{McGill University}
\email{hazem.hassan@mail.mcgill.ca}

\begin{abstract}
  Heegner cycles are higher weight analogues of Heegner points. Their arithmetic intersection numbers also appear as Fourier coefficients of modular forms and often belong to abelian extensions of imaginary-quadratic fields. Rotger and Seveso propose a precise recipe for the $p$-adic Abel-Jacobi images of  cycle classes whose existence is predicted by conjectures of Bloch and Beilinson and which would be a real-quadratic analogue to Heegner cycles: the Stark-Heegner cycles of the title. 

  In this paper, we generalize Darmon-Vonk's theory of rigid meromorphic cocycles to higher weight, producing a higher Green's pairing of real-quadratic divisors on the $p$-adic upper half-plane, which seems to be the real-quadratic analogue of the pairing of Heegner cycles. Computation of these values  for ``principal cycles'' gives evidence that they lie in abelian extensions of real-quadratic fields. The algebraicity of certain values of the higher Green's function is indirect evidence for the existence of algebraic Stark-Heegner cycles.
\end{abstract}
\maketitle
\vspace{-0.3in}
\tableofcontents

\newpage
\section*{Notation}
\begin{tabular}{p{3cm}p{10cm}p{1cm}}
    $k$ & an even positive integer & \\
    $n$ & $k-2$ & \\
    $p$ &  a prime number & \\
    $\Gamma$ & $\textrm{SL}_2(\mathbb{Z}[1/p])$ & \\
    $F$ & a finite field extension of $\mathbb{Q}_p$ & \\
    $\mathcal{O}_F$ &  the ring of integers of $F$ & \\
    $P_n(A)$ & polynomials of degree at most $n$ with coefficients in $A$ & \\ 
    $K$ & The unramified quadratic extension of $\mathbb{Q}_p$ & \\
    $\mathcal{H}_F$ & $p$-adic upper half-plane for a $p$-adic field $F$ & \pageref{def: upper half-plane} \\
    $\mathcal{T}_F$ & The Bruhat-Tits tree for $\mathcal{H}_F$ & \pageref{def: Bruhat-Tits tree}\\
    $\text{red}_F$ &  the reduction map $\mathcal{H_F} \rightarrow \mathcal{T}_F$ & \pageref{def: reduction map} \\
    $\mathcal{O}(k,\mathcal{H}_F)$ & Rigid analytic functions on $\mathcal{H}_F$ with a weight $k$ action & \pageref{sec: rigid analytic functions} \\ 
    $\mathcal{O}_\mathscr{L}(2-k,\mathcal{H}_F)$& log-rigid analytic functions on $\mathcal{H}_F$ with a weight $2- k$ action & \pageref{sec: rigid analytic functions} \\
    $\Sigma_{\mathscr{L}}(k,F)$ &  locally analytic functions on $\mathbb{P}^1(F)$ & \pageref{def: locally analytic functions}\\
    $\mathcal{M}(k ; K)$ & & \\  
    $\mathcal{M}_{\mathscr{L}}(2-k ; K)$ & & \\
    $\langle \cdot, \cdot \rangle_M$ &  Morita pairing &\pageref{sec: Morita and Breuil Duality} \\
    $\langle \cdot, \cdot \rangle_B$ & Breuil pairing &\pageref{thm: Breuil Duality} \\
    $C_F^1(A)$ & functions on oriented edges of $\mathcal{T}_F$ taking values in $A$ & \pageref{def: functions on tree}\\
    $C_F^0(A)$ & functions on vertices of 
    $\mathcal{T}_F$ taking values in $A$& \pageref{def: functions on tree}\\
    $C_{F,har}^1(A)$ & harmonic functions on oriented edges of $\mathcal{T}_F$ taking values in $A$ &\pageref{def: functions on tree} \\
    $\textrm{MS}^{\Gamma_0}(M)$ & $\Gamma_0$-invariant modular symbols valued in $M$ & \pageref{def: modular symbols} \\
    $\Sigma_{\tau}\{r,s\}$ & $\Gamma$-translates of $\tau$ whose geodesics intersect with $(r,s)$ & \pageref{def: orbits of rm points}\\
    $\text{Div}_{k,\mathcal{D}}$ & The divisor modular symbol associated to an RM-divisor $\mathcal{D}$ & \pageref{def: RM-divisors}\\
    $\text{Deg}_{k,\mathcal{D}}$ & modular symbol associated to an RM-divisor $\mathcal{D}$ valued in $C^0_{\mathbb{Q}_p}(P_n)$&\pageref{def: RM-divisors} \\
    $\Phi_{k,\mathcal{D}}$ & modular symbol associated to an RM-divisor $\mathcal{D}$ valued in $C^1_{K}(P_n)$ & \pageref{def: rm edge function}\\
    $J_{k,\mathcal{D}}$ & modular symbol associated to RM-divisor valued in weight-$k$ meromorphic functions& \pageref{def: deg 0 weight k cocycle}, \pageref{lem: weight k cocycle}\\ 
    $J_{k,\mathcal{D}}^{\mathscr{L}}$ & specific choice of modular symbol (not necessarily $\Gamma$-invariant) which lifts $J_{k,\mathcal{D}}$ & \pageref{wt-ncoLem}\\ 
    $\mathcal{J}_{k,\mathcal{D}}^{\mathscr{L}}$ & $\Gamma$-invariant modular symbol or cocycle class which lifts $J_{k,\mathcal{D}}$ in the situation where such a lift exists & \pageref{prop: principal values} \\
    $y_{k,\mathcal{D}},\; y_{k,\mathcal{D}}^\sharp$ & homology classes of divisors & \pageref{sec: higher Green's Functions}
    \end{tabular}
    \newpage

\section{Introduction}

\hfill \linebreak
\textbf{Complex Multiplication and Higher Green's Functions.}
The seminal papers of Gross--Zagier\cite{GZ} and Gross--Kohnen--Zagier\cite{GKZ} investigate the N\'eron--Tate height pairing of Heegner points on Jacobians of modular curves $X_0(N)$. The archimedean  contribution to this height pairing is realized as a limit of CM-values of automorphic functions: $$G_s(\tau,\sigma), \quad s>1 .$$ 
For positive integer values of $s$, the functions $G_s(z_1,z_2)$ are known as the higher Green's functions. These are real analytic functions on $\mathcal{H}\times \mathcal{H}$ which are invariant under the action of $\mathrm{SL}_2(\mathbb{Z})\times \mathrm{SL}_2(\mathbb{Z})$ and have a logarithmic singularity along the diagonal. In \cite{GZ} and \cite{GKZ} it is shown that the CM-values of higher Green's functions appear as Fourier coefficients of certain higher weight modular forms. Motivated by this observation, Gross, Kohnen and Zagier made two predictions: 
\begin{enumerate}

    \item The CM-values $G_{s}(\tau, \sigma)$ have an interpretation as the archimedean contribution to a global height pairing of some higher dimensional algebraic cycles.
    \item \textit{The Gross--Zagier algebraicity conjecture}. Certain ``principal'' linear combinations of CM-values $G_s(\tau,\sigma)$ are logarithms of algebraic numbers. 
\end{enumerate}

\hfill \linebreak
\textbf{Heegner Cycles.} The first prediction was addressed in \cite{zhang}. Let $\pi : \mathcal{E}\rightarrow X_0(N)$ denote the universal generalized elliptic curve. The Kuga--Sato variety, $\mathcal{Y}_k$, is the canonical resolution of the $(k-2)$-fold fibre product of $\mathcal{E}$ over $X_0(N)$. Associated to each CM-point, $\tau$, there is a Heegner cycle $$y_{k,\tau} \in \widehat{\mathrm{Ch}}^{k/2-1}(\mathcal{Y}_k),$$ which lies in the fiber over the $\tau$. S. Zhang showed that the archimedean contribution to the height pairing for two distinct Heegner cycles $y_{k,\tau}$ and $y_{k,\sigma}$ satisfies $$\langle y_{k,\tau}, y_{k,\sigma} \rangle_\infty = \frac{1}{2} G_{k/2}(\tau, \sigma).$$

Now to motivate the algebraicity conjecture of Gross and Zagier, consider a cycle $\mathcal{D}_1$ on an arithmetic variety $\mathcal{X}$. This cycle  is said to be principal if it vanishes in the Chow group of $\mathcal{X}$. This is accounted for by the existence of a formal linear combination of rational functions $$ \sum_i c_i (f_i,W_i)$$ whose divisor is $\mathcal{D}_1$. Given another cycle, $\mathcal{D}_2$, the archimedean contribution of the Arakelov height pairing $\langle \mathcal{D}_1,\mathcal{D}_2\rangle_\infty$ becomes a sum of logarithms $$\sum_i c_i \log |f_i(\mathcal{D}_2 \cdot W_i)|$$ i.e. it is a sum of logarithms of algebraic numbers belonging to the compositum of the fields of definition of $\mathcal{D}_1$ and $\mathcal{D}_2$.

This principle led Gross and Zagier to their conjecture that ``principal'' linear combinations of CM-values, $G_s(\tau, \sigma)$, are logarithms of algebraic numbers lying in the compositum of class fields of imaginary quadratic fields. In fact, they showed that when averaged over the Galois orbit of a pair of CM-points, principal linear combinations of the values \[\sum_{\theta \in \mathrm{Gal}(H_\tau\cdot H_{\sigma})}G_s(\tau ^\theta, \sigma ^\theta)\] yield sums of logarithms of rational numbers. Here $H_\tau$ and $H_\sigma$ denote the class fields of their respective CM point.
  
In \cite{zhang}, Zhang provided a proof of the algebraicity conjecture in the case where $\tau$ and $\sigma$ belong to the same imaginary quadratic field, which is conditional on the non-degeneracy of the height pairing. Mellit's thesis \cite{mellit} also uses an algebraic approach to give a proof of the algebraicity conjecture for $N=1$, $ k=4$ and $\tau = i$. 

However, a full unconditional proof of the algebraicity conjecture through this geometric interpretation seems to be currently out of reach. The main obstacles to this line of attack are deep conjectures on algebraic and arithmetic cycles such as Gillet--Soul\'e's arithmetic analogues of Grothendieck's  standard conjectures on algebraic cycles \cite{GS}. Concretely, these conjectures would imply that the action of Hecke operators on cycles factors through the Chow group, and that the arithmetic height pairing is non-degenerate.

Nevertheless, the algebraicity conjecture could be approached through the analytic machinery of regularized theta lifts. This direction was followed in the work of Viazovska \cite{Viazovska}, Bruinier--Ehlen--Yang \cite{BEY} and a full proof of the conjecture was completed in \cite{Li,BLY}. 

 \hfill\break
\textbf{Real Multiplication and Stark-Heegner Cycles.}
Let $p$ be a prime number, $E$ a real quadratic field in which $p$ is inert, and $H_c/E$ the narrow class field of conductor $c$. For any eigenform $f\in S_k\left(\Gamma_0(p)\right)$ and character $\chi$ of the group $\mathrm{Gal}(H_c/E)$ the sign of the functional equation for $L(f \otimes E, \chi ,s)$ is $-1$. Therefore, the order of vanishing of $L(f\otimes H_c,s)$ at $k/2$ is at least $|\mathrm{Gal}(H_c/E)|$. Thus the conjecture of Bloch and Beilinson predicts the existence of a collection of cycles in $\mathrm{Ch}^{k/2}(\mathcal{Y}_k)$ defined over $H_c$, which would account for the vanishing of the $L(f\otimes E,s)$ and would play the role of Heegner cycles in the real quadratic setting. 

In weight $2$, the emerging theory of Real Multiplication (RM) studies Stark--Heegner points as conjectural substitutes for Heegner points. In \cite{Darmon}, Darmon defines Stark--Heegner points through a $p$-adic integration theory which packages period integrals of modular forms into a measure and produces $p$-adic points on elliptic curves that are conjectured to be rational over abelian extensions of real quadratic fields. Similarly to their CM counterparts, Stark--Heegner points are related to Fourier coefficients of modular forms \cite{DT} and they appear to be non-torsion precisely when the derivative of a Hasse--Weil $L$-function does not vanish \cite{BD}. 

In the absence of a moduli theoretic interpretation for Stark--Heegner points, it remains unclear how to define the real quadratic analogues of Heegner cycles, which might be called Stark-Heegner cycles. Rotger and Seveso \cite{RS, Seveso} generalize Darmon's $p$-adic integration theory to higher weight and in doing so define local cohomology classes that they conjecture to be in the image of a global Selmer group associated to the Kuga--Sato varieties $\mathcal{Y}_k$. This conjecture was verified \cite{Seveso} for linear combinations of Stark-Heegner cycles by genus characters  An explicit geometric construction of Stark--Heegner cycles remains a mystery.

\hfill \linebreak
\textbf{Rigid Meromorphic Cocycles.} In the analogy between Complex Multiplication and Real Multiplication, the theory of rigid meromorphic cocycles is the counterpart to Green's functions on modular curves. Let $$\mathcal{H}_p \defeq \mathbb{P}^1(\mathbb{C}_p)\backslash \mathbb{P}^1(\mathbb{Q}_p)$$ be the $p$-adic upper half-plane. The action of $\Gamma \defeq \mathrm{SL}_2(\mathbb{Z}[1/p])$ on $\mathcal{H}_p$ by M\"{o}bius transformation is not discrete. It follows that if $\tau\in \mathcal{H}_p$ is a real-quadratic point, then there does not exist a $\Gamma$-invariant rigid meromorphic function on $\mathcal{H}_p$ whose divisor is supported on the $\Gamma$-orbit of $\tau$. Instead, Darmon--Vonk \cite{DV1} define a cocycle $$J_\tau  
\in Z^1(\Gamma, \mathcal{M}^\times/ \mathbb{C}_p^\times)$$
whose divisor is supported on the $\Gamma$-orbit of $\tau$. Here, $\mathcal{M}^\times$ denotes the $\Gamma$-module of rigid meromorphic functions on $\mathcal{H}_p$. When there is a lift of $J_\tau$ to $$\mathcal{J}_\tau \in H^1(\Gamma, \mathcal{M}^\times),$$ then we can evaluate $\mathcal{J}_{\tau}$ at another RM-point, $\sigma$, whose fundamental stabilizer is $\gamma_\sigma \in \Gamma$. This evaluation is given by $$J_p(\tau, \sigma) \defeq \mathcal{J}_\tau[\sigma] \defeq \mathcal{J}_\tau(\gamma_\sigma)(\sigma).$$  Darmon and Vonk have recently announced a proof of the  algebraicity of the norms of $J_p(\tau, \sigma)$ to $\mathbb{Q}_p$. Their proof also yields an explicit prime factorization formula for the norm of $J_p(\tau, \sigma)$, which is reminiscent of Gross--Zagier's factorization of singular moduli. The RM-values of rigid meromorphic cocycles behave analogously to the archimedean contribution of intersections of Heegner points; although the algebraicity of Stark--Heegner points is not known, one might speculate that the RM-values above encode some local contribution to a $p$-adic height pairing of Stark--Heegner points. 

For example, it is observed in \cite{DV1} that if $p=3$ and $\phi = \frac{1+\sqrt{5}}{2}$, then up to $100$ digits of $3$-adic accuracy \begin{equation}\label{jvalue1}J_3(\phi, 2\sqrt{5}) \stackrel{?}{=} \frac{174+832\sqrt{-1}}{2\cdot 5^2\cdot 17},\end{equation} and \begin{equation}\label{jvalue2}J_3(\sqrt{5}, 2\sqrt{5}) \stackrel{?}{=} \frac{-4623 + 136 \sqrt{-1}}{5^3 \cdot 37}.\end{equation}
Note that the narrow class field associated to the order $\mathcal{O}_{2\sqrt{5}}$ is $\mathbb{Q}(\sqrt{5},i)$.

\hfill\break
\textbf{Main Construction: $p$-adic Higher Green's Functions.}
In this paper, we generalize the theory of rigid meromorphic cocycles to higher weight, i.e. to cocycles of $H^1(\Gamma,\mathcal{M}(k))$, where $\mathcal{M}(k)$ is the module of rigid meromorphic functions with a weight $k$ action $$ f\mid_{k}\gamma = (cz + d)^{-k}f(\gamma z),\quad \text{where } \gamma = \begin{pmatrix} a & b \\ c&d \end{pmatrix}\in \Gamma .$$ We define the higher Green's functions as periods of such higher weight meromorphic cocycles.

We will classify the higher weight rigid meromorphic cocycles whose divisors are supported on the unramified points of $\mathcal{H}_p$. In fact, for any formal linear combination, $\mathcal{D}$,  of inert RM-points we will define a canonical cocycle $$J_{k,\mathcal{D}} 
\in H^1(\Gamma, \mathcal{M}(k)),$$ whose divisor is supported on the $\Gamma$-orbits of the points of $\mathcal{D}$. The requirement that the poles are supported on unramified points can be removed and a full classification of $$H^1_{par}(\Gamma, \mathcal{M}(k))$$ will be given in the author's thesis. 

The $p$-adic higher Green's function $G_k(\tau, \sigma)$ is defined as an RM-period of the cocycle $J_{k,\tau}$; there is a specific choice of cocycle $\mathcal{J}_{k,\tau}$ which takes values in locally analytic functions on $\mathcal{H}_p$ and satisfies $$d^{k-1}\mathcal{J}_{k,\tau} = J_{k,\tau}.$$ The higher Green's function is then equal to $$G_k(\tau,\sigma)  = \delta_{2-k}^{(k-2)/2}\left(\mathcal{J}_{k,\tau}(\gamma_\sigma)\right)(\sigma),$$ where $\delta_i = \frac{d}{dz} + \frac{i}{z-\overline{z}}$ is the usual raising operator applied formally to the function $\mathcal{J}_{k,\tau}(\gamma_\sigma)$, and conjugation is the non-trivial automorphism of $\mathbb{Q}(\sigma)$. When $k = 2$, the Green's function recovers Darmon--Vonk's real-quadratic singular moduli: $$G_2\left([\tau],[\sigma]\right) = \log_p \left(J_p(\tau,\sigma)\right).$$

\hfill\break
\textbf{Main Conjecture and Supporting Evidence.}
We implemented an algorithm to compute some values of $G_k$ on the computer algebra system SageMath.
The resulting values of $G_k$ at ``principal'' linear combinations of RM-orbits appear to be $p$-adic logarithms of algebraic numbers lying in the compositum of class fields of the real quadratic fields associated to the RM-orbits.

For example, let $p=3$ and $k=4$. In this case, $S_4(\Gamma_0(3))$ is trivial, which implies that every divisor is principal. Let $\phi = \frac{1+\sqrt{5}}{2}$ be the golden ratio, then up to $140$ digits of $3$-adic accuracy we compute \begin{align*} & 40 G_4(9[\phi] - 2[\sqrt{5}], [2\sqrt{5}])\\ & \stackrel{?}{=}  5\log_p\left(\frac{2-i}{2+i}\right)+17\log_p\left(\frac{6+i}{6-i}\right)+63\log_p\left(\frac{4+i}{4-i}\right)+26\log_p\left(\frac{5+i}{5-i}\right). \end{align*}

Note that $i$ generates the narrow class field of the order $\mathcal{O}_{2\sqrt{5}}$ over $\mathbb{Q}(\sqrt{5})$. The prime factors in the factorization above are all inert in $\mathbb{Q}(\sqrt{5})$. Further, apart from the prime $13$, all the other primes appear either in the factorization of $J_{\phi}[2\sqrt{5}]$ or  $J_{\sqrt{5}}[2\sqrt{5}]$ given in (\ref{jvalue1}) and (\ref{jvalue2}).

\begin{conjecture}\label{conj: intro alg}
Let $$\mathcal{D} = \sum_i \tau_i$$ be a formal linear combination of RM-points, which is principal i.e. $\mathrm{Obs}_0(J_{k,\mathcal{D}})$ (Definition \ref{def: log obstruction}) is trivial. Then for any other RM-points $\sigma$, the value of the Green's function $G_k(\mathcal{D}, \sigma)$ is a linear combination of logarithms of algebraic numbers belonging to the compositum of narrow class fields associated with the quadratic orders $\mathcal{O}_{\tau_i}$ and $\mathcal{O}_{\sigma}$.
\end{conjecture}

The norms of the algebraic numbers that appear in Conjecture \ref{conj: intro alg} also seem to admit factorization formulas similar to those of singular moduli.  The primes appearing in the factorization of the norm of $G_k(\tau,\sigma)$ are not split in either of the two orders, and their multiplicities seem to be determined by weighted intersection numbers of geodesics on Shimura curves associated to definite quaternion algebras. Conjecture \ref{conj: factorization} has a more precise formulation of this observation.

The algebraicity of the values $G_k(\tau, \sigma)$ lends evidence to the conjectured existence of rational Stark--Heegner cycles and further adds to the growing analogy between the theories of Real Multiplication and complex multiplication. 

\hfill \break
\textbf{Meromorphic Cocycles of Higher Weight.}
In the multiplicative setting of \cite{DV1}, the meromorphic cocycles are defined explicitly as infinite products which are proven to converge by hand. In higher weight, the analogous infinite sums need not converge. The key idea to remedy this is as follows. If $f$ is a rigid meromorphic function on $\mathcal{H}_p$ whose poles all lie in a finite extension $E/\mathbb{Q}_p$, then the restriction $f\vert_{\mathcal{H}_E}$ is a rigid analytic function on $\mathcal{H}_E \defeq \mathbb{P}^1(\mathbb{C}_p)\backslash \mathbb{P}^1(E)$. The problem of defining meromorphic functions on $\mathcal{H}_p$ with prescribed poles now translates to finding certain harmonic boundary measures on $\mathbb{P}^1(E)$.

For example, the rigid meromorphic period function $j_\tau$ is defined by
$$j_{\tau} \defeq \prod_{w \in \Gamma \cdot \tau} t_w(z)^{\delta_{\infty}(w)}, \quad  \text{where }  t_w(z) =\begin{cases}  z-w & |w|\leq 1, \\ z/w-1 & |w|>1.  \end{cases}$$ for a certain function $\delta_\infty(w) \in \{-1,0,1\}$. The logarithmic derivative, $$\textrm{dlog}(j_\tau) = \sum_{w\in \Gamma \cdot \tau} \frac{\delta_{\infty}(w)}{z-w},$$ belongs to the module, $\mathcal{M}(2)$, of weight $2$ rigid meromorphic functions and has poles only at points in the orbit of $\tau$. The boundary measure associated to this function is then simply the indicator function at each $w$ in the orbit $\Gamma\cdot \tau$, weighted by $\delta_\infty(w)
$.

In \cite{Neg1}, Negrini shows that any parabolic cocycle in $Z^1_{par}(\Gamma, \mathcal{M}(k))$ has its poles supported only on a finite union of $\Gamma$-orbits of RM-points, and she produced cocycles of rigid meromorphic functions whose poles are supported on the union of $\Gamma$-orbits of an RM-point $\tau$ and its conjugate $\overline{\tau}$. The boundary measures described above will produce parabolic cocycles whose divisors are supported on the orbit of $\tau$ only, and therefore allows us to complete the classification of parabolic cocycles. 

Another difference in the higher weight setting is the lack of an analogue to the logarithmic derivative map $$ \textrm{dlog}:\mathcal{M}^\times \longrightarrow \mathcal{M}(2).$$ 

Real quadratic singular moduli are associated to the multiplicative cocycles. In higher weight, we consider instead the modules $\mathcal{M}_{\mathcal{L}}(2-k)$ defined to be the set of Coleman primitives of functions $f\in \mathcal{M}(k)$ under iterated differentiation $(k-1)$-times and with a choice of $p$-adic logarithm such that $\log_{\mathscr{L}}(p) = \mathscr{L}$. Note that the iterated differentiation map $$d^{k-1}: \mathcal{M}_{\mathscr{L}}(2-k) \longrightarrow \mathcal{M}(k)$$ is a $\Gamma$-module surjection. The spaces $\mathcal{M}_{\mathscr{L}}(2-k)$ are the meromorphic counterpart to those considered by Breuil \cite{Breuil2} \cite{Breuil1}. In weight $2$, the image of $\mathcal{M}^\times$ under the logarithm, $\log_\mathscr{L}$, is dense in $\mathcal{M}_{\mathscr{L}}(0)$.

Finally, to define $G_k(\tau,\sigma)$ we use Breuil duality, to pair the meromorphic cocycle $$J^{\mathscr{L}}_{\tau} \in H^1(\Gamma, \mathcal{M}_{\mathscr{L}}(2-k))$$ on $\mathcal{H}_p$ with a locally analytic functions on the boundary of the form $$\frac{(t-\sigma)^{(k-2)/2}(t-\overline{\sigma})^{(k-2)/2}}{(\sigma-\overline{\sigma})^{(k-2)/2}}\log_{\mathscr{L}}(t-\sigma).$$

\hfill\break

\vspace{0.2in}
\noindent \textbf{Acknowledgements:} This work is part of my PhD thesis. I would like to thank my advisor Henri Darmon for introducing me to this topic and for guiding me along the way. I am supported by an FRQNT Doctoral Training Scholarhip \href{https://doi.org/10.69777/363239}{https://doi.org/10.69777/363239}. I was previously partially supported by a Graduate Student Scholarship of the Institut des Sciences Mathématiques au Qu\`ebec.

\section{ The $p$-adic upper half-plane}
This section introduces $p$-adic analysis on the Drinfeld upper half-plane, which serves as the  scaffolding for the rest of the paper. We will define spaces of rigid analytic functions, as well as spaces of their Coleman primitives. Then, we will discuss Morita and Breuil duality, which relate functions on the upper half-plane to distributions on the boundary. Much of the information of this section is presented in more detail in \cite{DT}.

Fix a prime number $p$. Let $F$ be a fixed finite extension of $\mathbb{Q}_p$. Denote by $\mathcal{O}_F$ its ring of integers and $\kappa(F)$ its residue field, which has size $q_F$. Fix a choice of uniformizer $\pi_F\in \mathcal{O}_F$. 
\begin{definition}[The $p$-adic upper half-plane]\label{def: upper half-plane}

The $p$-adic upper half plane over $F$ is the rigid analytic space, $\mathcal{H}_F$, whose $E$-rational points are $$\mathcal{H}_F(E) \defeq \mathbb{P}^1(E)\backslash \mathbb{P}^1(F),$$ for any complete extension $E/F$. In particular, $\mathcal{H}_F(\mathbb{C}_p) = \mathbb{P}^1(\mathbb{C}_p)\backslash\mathbb{P}^1(F)$.

\end{definition}

The group $\textrm{GL}_2(F)$ acts on $\mathcal{H}_F$ by M\"{o}bius transformation $$\begin{pmatrix} a & b \\ c& d\end{pmatrix}\cdot z \mapsto \frac{az+b}{cz+d}.$$
For any complete extension $E$ of $F$, define the closed disks in $\mathbb{P}^1(E)$ by $$B_E\left([x_0:x_1],r\right) = \left\{ [y_0: y_1] \mid \mathrm{ord}_p\left( x_0 \cdot y_1 - x_1 \cdot y_0\right) \geq r, \; \; (y_0,y_1)\text{ is a primitive pair}\right\}.$$ A pair $(y_0,y_1)$ is said to be primitive if both $y_0$ and $y_1$ belong to $\mathcal{O}_E$ and at least one of them is a unit.

\begin{definition}[Bruhat--Tits tree]\label{def: Bruhat-Tits tree}
The Bruhat--Tits tree over $F$ is the graph, $\mathcal{T}_F$, whose vertex set is the set of homothety classes of $\mathcal{O}_F$-lattices in $F^2$. A pair of classes of lattices $[L_1]$ and $ [L_2]$ are related by an edge if there exist representatives $L_1$ and $L_2$, respectively, such that $$ \pi_F L_1 \subsetneq L_2 \subsetneq L_1.$$ In this case, $[L_1]$ and $[L_2]$ are said to be $q_F$-neighbors.

We will denote the set of vertices of $\mathcal{T}_F$ by $\mathcal{T}_F^0$, the set of oriented edges by $\mathcal{T}_F^1$ and the set of unoriented edges by $\overline{\mathcal{T}_F^1}$. The group $\textrm{GL}_2(F)$ acts on $\mathcal{T}_F$ via its natural action on lattices in $F^2$. We will denote oriented edges as an ordered pair, $(v,w)$, where $v$ and $w$ are neighbouring vertices.
\end{definition}

$\mathcal{T}_F$ is a tree, and it is  $(q_F + 1)$-regular; that is, each vertex has precisely $q_F + 1 = |\mathbb{P}^1(\kappa(F))|$ neighbours. The choice of the standard basis on $F^2$ privileges distinguished elements of the $\mathcal{T}_F$. Namely, define $$v_0 \defeq \left[\mathcal{O}_F^2 \right], \quad v_1 \defeq \left[ \mathcal{O}_F \oplus \pi_F \mathcal{O}_F \right].$$ We will refer to $v_0$ as the standard vertex. Observe that $v_0$ and $v_1$ are adjacent and we will denote the edge $(v_0, v_1)$ by $e_0$ and refer to it as the standard edge.

The \textit{level} of a vertex, $v=[L]$, is its distance from $v_0$ on the tree. Equivalently, the level is the smallest non-negative integer $l$ such that there exists a representative, $L$, which satisfies $$\pi_F^l \mathcal{O}_F^2 \subset L \subset \mathcal{O}_F^2.$$ We will call such a representative an optimal representative. The level of an edge is the level of its vertex which is further from $v_0$.

For any vertex, $v = [L]$ of $\mathcal{T}_F$ with optimal representative $L$, let $L^{\mathrm{prim}}$ denote the primitive elements of $L$, that is the elements of the lattice such that at least one of the two coordinates is a unit in $\mathcal{O}_F$. The image of the projection $$L^{\mathrm{prim}} \longrightarrow\mathbb{P}^1(\mathcal{O}_F) = \mathbb{P}^1(F)$$ is a closed disk, which we will denote by $U_v$. The disk $U_v$ is independent of the choice of optimal representative.

The set of ends of $\mathcal{T}_F$ is the set of equivalence classes of infinite non-backtracking sequences of adjacent vectors, $$(u_0, u_1, \cdots),$$ modulo tail-equivalence; that is; two sequences $(u_0, u_1,\cdots)$ and $(w_0, w_1 ,\cdots)$ are equivalent if there exists an integer $m$ such that $u_i = w_{i+m}$ for all $i$ large enough.

\begin{proposition}
	The set of ends of $\mathcal{T}_F$ can be naturally identified with $\mathbb{P}^1(F)$. 
\end{proposition}

\begin{proof}

	Without loss of generality, we can assume that each end $(w_0,w_1,\cdots)$ begins at the standard vertex. Then, the intersection $$\bigcap_{i\in \mathbb{N}} U_{w_i}$$ is a single element in $\mathbb{P}^1(F)$. In fact, every element $ \tau \in \mathbb{P}^1(F)$ can be obtained in this way for a unique end; the choices of optimal representatives of the path from $v_0$ to $\tau$ can be obtained from the reduction of $\tau$ in $\mathbb{P}^1(\mathcal{O}_F/\pi_F^i \mathcal{O}_F)$.
\end{proof}

For any vertex $v$, the open set $U_v$ is precisely the ends of $\mathcal{T}_F$ of non-backtracking paths which begin at $v_0$ and pass through $v$. Similarly, to each oriented edge, $e = (w_0, w_1)$, we let $U_e$ be the set of ends of non-backtracking paths which begin at $w_0$ and pass through $w_1$.

\begin{definition}[Reduction Map] \label{def: reduction map}
There is a $\textrm{GL}_2(F)$-equivariant map $$ \red_F: \mathcal{H}_F(\mathbb{C}_p) \longrightarrow\mathcal{T}_F^0\cup \overline{\mathcal{T}_F^1}$$ called the reduction map, which is defined as follows:  Given $\tau\in \mathcal{H}_F(\mathbb{C}_p)$, there exists a minimal $r \in \mathbb{R}$ such that the intersection $B_{\mathbb{C}_p}(\tau, r) \cap \mathbb{P}^1(F)$ is a non-empty closed disk in $\mathbb{P}^1(F)$ and is equal to $U_v$ for a unique vertex $v$. 

\begin{itemize}
	\item If $r\in \mathrm{ord}_p(F)$, then set $\red_F(\tau) \defeq v$,
	\item otherwise, let $e$ be the unique oriented edge starting at $v$ and pointing towards the standard vertex. Then we set $\red_F(\tau)$ to be the unoriented edge with the same vertices as $e$. 
\end{itemize}

\end{definition}

\begin{remark}
 Given an oriented edge, $e$, we will write $\mathrm{red}_F^{-1}(e)$ to denote the preimage of the unoriented edge with the same vertices as $e$.

\end{remark}

Given a vertex $v\in \mathcal{T}_F^0$, its preimage, $\red_F^{-1}(v)$, is an admissible affinoid of $\mathcal{H}_F$; the preimage of $v_0$ is referred to as the standard affinoid. The preimage of an edge is an annulus and the choice of an orientation on the edge corresponds to an orientation on the annulus.

Let $\mathcal{T}_F^{\leq l}$ be the subtree of $\mathcal{T}_F$ consisting of vertices and edges of level at most $l$. The preimage of this subtree, $\mathcal{H}_F^{\leq l} \defeq \red_F^{-1}\left(\mathcal{T}_F^{\leq l}\right)$ is an admissible affinoid, and the collection $\left\{ \mathcal{H}_F^{\leq l}\right\}_l$ is an admissible covering of $\mathcal{H}_F$.

\subsection{Rigid Analytic Functions}\label{sec: rigid analytic functions}
The ring of rigid analytic functions on $\mathcal{H}_F$ is a projective limit of affinoid algebras of entire functions on the admissible covering $$\bigcup_{l=0}^\infty \mathcal{H}_F^{\leq l}.$$

We begin by giving a more explicit description of $\mathcal{H}_F^{\leq l}$. Fix a set of representatives, $\mathcal{R}_l$, of $\mathbb{P}^1(F) = \mathbb{P}^1(\mathcal{O}_F)$ modulo $\pi_F^l$ such that $\infty\in\mathbb{P}^1(F)$ is the representative for $\infty\;  (\mathrm{mod}\; \pi_F^l)$. Define the open disks 

$$B^{-}_{\mathbb{C}_p}\left([x_0:x_1],r\right) = \left\{ [y_0: y_1]\in \mathbb{P}^1(\mathbb{C}_p) \;| \; \mathrm{ord}_p\left( x_0 \cdot y_1 - x_1 \cdot y_0\right) > r \right\},$$ where, as before, all coordinates are taken to be primitive. Then the affinoids can be described as $$\mathcal{H}_F^{\leq l} = \mathbb{P}^1(\mathbb{C}_p)\backslash \bigcup_{a\in \mathcal{R}_l}B_{\mathbb{C}_p}^{-}(a,l).$$

Let $\mathcal{O}(\mathcal{H}_F^{\leq l})$ be the $\mathbb{C}_p$-algebra of functions of the form $$\sum_{i = 0}^\infty b_{\infty, i}z^i + \sum_{a\in \mathcal{R}_l\backslash\{\infty\}}\sum_{i=1}^\infty\frac{b_{a,i}}{(z-a)^i},$$
with $b_{\infty,i}$, $b_{a,i}\in\mathbb{C}_p$ satisfying $\lim_{i\to \infty}p^{-il}b_{a, i} = 0$ for every $a\in \mathcal{R}_l$. These are precisely the sums of power series which converge away from the disks $B^-_{\mathbb{C}_p}(a,l)$.

The $\mathbb{C}_p$-algebra of rigid analytic functions on $\mathcal{H}_F$ is the projective limit \[\mathcal{O}(\mathcal{H}_F) \defeq \varprojlim_l \mathcal{O}(\mathcal{H}_F^{\leq l}),\] where the projective morphisms are given by restriction of functions. The field of meromorphic functions $\mathcal{M}(\mathcal{H}_F)$ is the fraction field of $\mathcal{O}(\mathcal{H}_F)$.

 For $k$ any even integer and $n = k-2$, let $\mathcal{O}(k, \mathcal{H}_F)$ be the right $\textrm{GL}_2(F)$-module whose underlying vector space is $\mathcal{O}(\mathcal{H}_F)$, and is equipped with a weight $k$ action $$\left(f|_k\gamma\right)(z) \defeq \frac{\mathrm{det}(\gamma)^{k/2}}{(cz+d)^k} f(\gamma\cdot z), \; \; \text{ where }  \gamma = \begin{pmatrix} a & b \\ c & d \end{pmatrix}\in \textrm{GL}_2(F). $$ 

For a positive weight $k$, differentiating $(k-1)$ times is a morphism of $\textrm{GL}_2(F)$-modules $$\mathcal{O}(2-k,\mathcal{H}_F) \xrightarrow{d^{k-1}} \mathcal{O}(k,\mathcal{H}_F).$$ The kernel of this map is the $\textrm{GL}_2(F)$-module of polynomials of degree at most $n$ with coefficients in $\mathbb{C}_p$, which we will denote by $P_{n}(\mathbb{C}_p)$.

Following Breuil\cite{Breuil1, Breuil2}, we will define spaces of Coleman primitives of rigid-analytic functions. Fix a choice of $\mathscr{L}\in\mathbb{C}_p$; this determines a branch of logarithm $\log_\mathscr{L}:\mathbb{C}_p^\times\longrightarrow\mathbb{C}_p$ such that $$\log_\mathscr{L}(p) = \mathscr{L}.$$

For a positive weight $k>0$, the space $\mathcal{O}_{\mathscr{L}}(2-k, \mathcal{H}_F)$ is defined as a projective limit. For the affinoids $\mathcal{H}_F^{\leq l}$ define $\mathcal{O}_{\mathscr{L}}(2-k, \mathcal{H}_F^{\leq l})$ to be the $\mathbb{C}_p$ vector space of functions of the form $$\sum_{i = 0}^\infty b_{\infty, i}z^i + \sum_{a\in \mathcal{R}_l\backslash\{\infty\}}\sum_{i=1}^\infty\frac{b_{a,i}}{(z-a)^i} + \sum_{a\in \mathcal{R}_l\backslash\{\infty\}} P_a(z)\log_{\mathscr{L}}(z-a),$$ where $P_a$ are polynomials in $P_n(\mathbb{C}_p)$ and $b_{\infty,i}, b_{a,i}\in\mathbb{C}_p$ satisfying $\lim_{i\to \infty}p^{-ir}b_{a, i} = 0$ for every $a\in \mathcal{R}_l$. 
The space $\mathcal{O}_{\mathscr{L}}(2-k, \mathcal{H}_F)$ is equipped with a $\textrm{GL}_2(F)$-action with weight $2-k$. 

\begin{proposition}[\cite{Breuil1} Proposition 3.2.2]
There is an exact sequence of $\textrm{GL}_2(F)$-modules $$0 \longrightarrow P_n(\mathbb{C}_p) \longrightarrow\mathcal{O}_{\mathscr{L}}(2-k, \mathcal{H}_F) \xrightarrow{d^{k-1}} \mathcal{O}(k, \mathcal{H}_F)\longrightarrow0.$$
\end{proposition}

It will be convenient to also view rigid analytic functions as global sections of a local system on $\mathcal{H}_F$. 

\begin{proposition}
	There are $\textrm{GL}_2(F)$-module morphism $$\beta: \mathcal{O}_{\mathscr{L}}(2-k, \mathcal{H}_F)\longrightarrow\mathcal{O}_{\mathscr{L}}(0, \mathcal{H}_F)\otimes P_{n}(\mathbb{C}_p)$$ and $$\gamma: \mathcal{O}(k, \mathcal{H}_F)\longrightarrow\mathcal{O}(2,\mathcal{H}_F)\otimes P_{n}(\mathbb{C}_p).$$ 
	such that the following diagram commutes 
	\begin{center}
	\begin{tikzcd}
		\mathcal{O}_{\mathscr{L}}(2-k, \mathcal{H}_F)\arrow{r}{d^{k-1}} \ar{d}{\beta} & \mathcal{O}(k, \mathcal{H}_F) \ar{d}{\gamma}\\ 
		\mathcal{O}_{\mathscr{L}}(0, \mathcal{H}_F)\otimes P_{n}(\mathbb{C}_p)\ar{r}{d} & \mathcal{O}(2,\mathcal{H}_F)\otimes P_{n}(\mathbb{C}_p)
	\end{tikzcd}
	\end{center}

\end{proposition}

\begin{proof}
The morphisms are given by $$\gamma: g(z)\mapsto g(z)(T-z)^{n}$$
 and 
$$\beta: G(z)\mapsto n!\sum_{i=0}^{n}\frac{G^{(i)}(z)}{i!}(T-z)^i.$$

It is a straightforward, but tedious, computation to verify that these functions are $\textrm{GL}_2(F)$-morphisms and that the diagram commutes.

\end{proof}

The space $\mathcal{O}(2,\mathcal{H}_F)\otimes P_{n}(\mathbb{C}_p)$ has a natural filtration of $\textrm{GL}_2(F)$-modules given by $$ F^j \defeq \mathrm{span}_{\mathcal{O}(2,\mathcal{H}_F)} \left\langle (z-T)^i \mid i \geq j \right\rangle. $$ The map $\gamma$ identifies $\mathcal{O}(k,\mathcal{H}_F)$ with $F^{n}$. There is also a filtration on $\mathcal{O}_{\mathscr{L}}(0,\mathcal{H}_F)\otimes P_{n}(\mathbb{C}_p)$ defined by $$F_{\mathscr{L}}^j\defeq d^{-1}(F^j). $$ The map $\beta$ induces an isomorphism between $\mathcal{O}_{\mathscr{L}}(2-k,\mathcal{H}_F)$ and $F_{\mathscr{L}}^{n}$.

\subsection{Locally Analytic Functions on the Boundary}
\begin{definition}\label{def: locally analytic functions} Let $k\geq 2 $ be an even weight. 
	\begin{itemize}
 \item The space of locally analytic functions, $C(k,F)$, is the space of functions $$F\longrightarrow\mathbb{C}_p,$$ such that for each $a\in F$ there exists a neighbourhood of $a$ on which $f$ can be expressed as a convergent power series $$f(t) = \sum_{i\geq 0}c_{a,i}(t-a)^i,$$ and in some neighbourhood of $\infty$, the function $f$ can be expressed as a convergent (away from $\infty$) power series $$f(t) = \sum_{i\leq{n}}c_{\infty,i}t^i;$$ i.e. it has a pole at $\infty$ of order at most $n$.
 \item Fix a branch cut of $p$-adic logarithm, $\mathscr{L}$. The space $C_{\mathscr{L}}(k,F)$ is the space of locally analytic functions on $F$, which in a neighbourhood of $\infty$ can be expressed in the form \begin{equation} \label{logpoly} f(t) = \left(\sum_{i\leq n}c_{\infty,i}t^i\right) - 2P(t)\log_{\mathscr{L}}(t),\end{equation} where $P$ is any polynomial of degree at most $n$. 
	\end{itemize}
\end{definition}

$C(k, F)$ is a subspace of $C_{\mathscr{L}}(k,F)$, and we equip both spaces with an action of $\textrm{GL}_2(F)$ by $$(f|\gamma)(t) = \frac{(cz+d)^{n}}{\det(\gamma)^{n/2}}\left[f(\gamma\cdot t)   + P(\gamma\cdot t)\log_{\mathscr{L}}\left(\frac{\det(\gamma)}{(cz+d)^2}\right)\right];$$ using the notation in (\ref{logpoly}).

\begin{definition}
Observe that $C(k,F)$ has $P_{n}(\mathbb{C}_p)$ as a sub-module.  Define $\Sigma(k,F)$ and $\Sigma_{\mathscr{L}}(k,F)$ to be the quotient modules $$\Sigma(k,F) \defeq C(k,F)/P_{n}(\mathbb{C}_p),$$ and $$\Sigma_{\mathscr{L}}(k,F) \defeq C_{\mathscr{L}}(k,F)/P_n(\mathbb{C}_p).$$

\end{definition}

\begin{proposition}
The following is an exact sequence $$0\longrightarrow\Sigma(k,F)\longrightarrow\Sigma_{\mathscr{L}}(k,F) \longrightarrow P_{n}(\mathbb{C}_p)\longrightarrow0,$$ 

where the map $\Sigma_{\mathscr{L}}(k,F) \longrightarrow P_{n}(\mathbb{C}_p)$ is given by sending to $f(t)\mapsto P(t)$ from \emph{(}\refeq{logpoly}\emph{)}.

\end{proposition}

We will write $\Sigma$ or $\Sigma(F)$ for $\Sigma(2,F)$ and also $\Sigma_{\mathscr{L}}$ or $\Sigma_{\mathscr{L}}(F)$ for $\Sigma_{\mathscr{L}}(2,F)$.
Define $\Phi_{\mathscr{L}}^j$ to be the subspace of $\Sigma_{\mathscr{L}}\otimes P_{n}(\mathbb{C}_p)$ consisting of functions which can locally, on $F$, be written in the form $(t-T)^j g(t)$ where $g(t)$ is an analytic function. Let $\Phi^j$ be the intersection $\Phi^j_{\mathscr{L}}\cap \Sigma\otimes P_{n}(\mathbb{C}_p)$.

The map induced from $$f(t)\otimes P(T)\mapsto f(t)P(t)$$ gives an isomorphism \begin{equation} \label{equation: filtration iso boundary}\Phi^0(\Sigma_\mathscr{L}\otimes P_n )/ \Phi^1 (\Sigma_\mathscr{L}\otimes P_n) \cong \Sigma_\mathscr{L}(k,F). \end{equation}

\subsection{Morita and Breuil Duality}\label{sec: Morita and Breuil Duality}

There is a duality between rigid analytic functions on $\mathcal{H}_F$ and locally analytic functions on $\mathbb{P}^1(F)$. Given  $f\in \mathcal{O}(2,\mathcal{H}_F)$ and $g\in \Sigma(0, F)$. Pick $l$ large enough so that $f|\mathcal{H}_F^{\leq l}$ can be written as $$f(z) = \sum_{i = 0}^{\infty} b_{\infty, i}z^{-i} + \sum_{a\in \mathcal{R}_l\backslash\{\infty\}}\sum_{i=1}^\infty\frac{b_{a,i}}{(z-a)^i},$$ and for each $a\in \mathcal{R}_l$, the function $g$ has a power series expansion of the form $$\sum_{j = 0}^\infty c_{a,j}(z- a)^j $$ on the disk $B_F(a,l)$ for every $a\neq \infty$, and of the form
$$ \sum_{j = 0}^{\infty} c_{\infty, j}z^{-j}$$ on the disc $B_F(\infty,l)$.

Then the residue pairing is defined by \begin{align*}&\langle f, g \rangle_M \defeq \sum_{\sigma \in \mathbb{P}^1(F)} \mathrm{res}_\sigma (f(z)g(z)dz ) \\ &\defeq \sum_{a\in \mathcal{R}_l\backslash \{\infty\}}\left( \sum_{i = 1}^{\infty} b_{a,i}c_{a,i-1} + b_{a,1}c_{\infty,0} \right) - \sum_{i = 1}^{\infty} b_{\infty, i}c_{\infty,i+1} . \end{align*}

\begin{theorem}[Morita Duality \cite{morita}]
The pairing $$\langle \cdot, \cdot \rangle_M  : \; \mathcal{O}(2,\mathcal{H}_F) \times \Sigma(2,F) \longrightarrow\mathbb{C}_p$$ defines a perfect bilinear pairing of $\textrm{GL}_2$-modules. 
\end{theorem}

\begin{remark}Given a point $\tau \in \mathcal{H}_F$, then the function $\displaystyle{\frac{1}{t-\tau}}$ belongs to $\Sigma(2,F)$ and the Morita pairing recovers values of $f$ by  \begin{equation}\label{moritaover}\left\langle f(z), \frac{1}{t-\tau} \right\rangle _M  = f(\tau).\end{equation}
Further, given two points $\tau_1, \tau_2 \in \mathcal{H}_F$, the function $$\log_{\mathscr{L}}\left(\frac{t-\tau_1}{t-\tau_2}\right)$$ also belongs to $\Sigma(2,F)$. We also have \begin{equation}\label{moritalog}\left\langle f(z), \log_{\mathscr{L}}\left(\frac{t-\tau_1}{t-\tau_2}\right)\right\rangle _M =  \int_{\tau_1}^{\tau_2}f(z)\; dz,\end{equation} where the integral is a Coleman integral \cite{Coleman}, with the same choice of branch of logarithm.

\end{remark}
\begin{proposition}
	For every $n$, there is a perfect pairing $$ \langle \cdot ,\cdot \rangle_{P_n} : P_n(\mathbb{Q})\times P_n(\mathbb{Q}) \longrightarrow\mathbb{Q}.$$
\end{proposition}

\begin{proof}
	Fix an isomorphism $\mathrm{Sym}^n(\mathbb{Q}^2) \cong P_n(\mathbb{Q})  $ by mapping $$ \begin{pmatrix} 1 \\ 0 \end{pmatrix} \mapsto 1,$$ and $$ \begin{pmatrix} 0 \\ 1 \end{pmatrix} \mapsto T.$$ The determinant gives a pairing on $\mathbb{Q}^2$ which induces a pairing on the symmetric product. Composing by this isomorphism, the pairing on $P_n$ is given by $$ \left\langle T^i ,T^j \right\rangle _{P_n} = (-1)^i \binom{n}{i}^{-1} \delta_{n, i+j}.$$

\end{proof}

It follows that for every $i$ and $a\in \mathbb{Q}$,  $$\langle (t-a)^i,P(t) \rangle_{P_n} = (-1)^i \frac{i!}{n!}P^{(n- i)}(a),$$ where $P^{(s)}$ denotes differentiation $s$ times.

The combination of the Morita pairing and the pairing on polynomials produces a pairing $$\mathcal{O}(2,\mathcal{H}_F)\otimes P_{n}(\mathbb{C}_p) \times \Sigma(2, F)\otimes P_{n}(\mathbb{C}_p)\longrightarrow\mathbb{C}_p,$$ which we will also denote by $\langle \cdot , \cdot  \rangle_M$. In terms of the filtration, the pairing descends to a perfect pairing (\cite{dSBounded} Corollary 2.5)$$F^{i}/F^{i+1} \times \Phi^{n-i}/\Phi^{n+1-i} \longrightarrow\mathbb{C}_p.$$ In particular, with $i=0$ and  there is a duality $$\mathcal{O}(k,\mathcal{H}_F) \times \Sigma(k,F)\longrightarrow\mathbb{C}_p,$$ which we will again denote by $\langle \;\cdot\; , \cdot\; \rangle_M$. 
 
Similarly to (\ref{moritaover}), for any $f\in \mathcal{O}(k,\mathcal{H}_F)$ we have  $$\left\langle f(z), \frac{1}{t-\tau}\right\rangle_M = f(\tau).$$

To generalize (\ref{moritalog}), which relates the Morita pairing to Coleman integrals, we have to consider integrals on the local system $P_n(\mathbb{C}_p)$; for $f\in \mathcal{O}(k,\mathcal{H}_F)$ and $\tau_1, \tau_2 \in \mathcal{H}_F$  and $P\in P_n(\mathbb{C}_p)$ we have

 $$\left\langle f(z), P(t)\log_\mathscr{L}\left(\frac{t-\tau_1}{t-\tau_2}\right) \right\rangle _M =\left\langle\int_{\tau_1}^{\tau_2} \beta(f)\; dz,P(T)\right\rangle_{P_n}.$$

\begin{theorem}[Breuil Duality, \cite{Breuil1} Theorem 1.1.4] \label{thm: Breuil Duality}
There is a unique pairing of $\textrm{GL}_2$-modules $$\langle \cdot ,\cdot \rangle_B : \mathcal{O}_{\mathscr{L}} \times \Sigma_{\mathscr{L}} \longrightarrow\mathbb{C}_p, $$  
such that  \begin{enumerate}
\item If $f\in \mathcal{O}_{\mathscr{L}}$ and $\tau \in \mathcal{H}_p$ $$\langle f,\log_\mathscr{L}(t-\tau)\rangle_B = f(\tau),$$
\item If $g\in \Sigma \subset \Sigma_{\mathscr{L}},$ then $$\langle f, g \rangle_B = \left\langle \frac{d}{dz}\left(f\right), g \right\rangle_M,$$

\item If $f\in \mathcal{O}$, then $$\langle f, g\rangle_B = \left\langle f, \frac{d}{dt}(g)\right\rangle_M.$$
\end{enumerate}

\end{theorem}

As before, combining the Breuil pairing with the pairing on polynomials we obtain a paring $$\left(\mathcal{O}_{\mathscr{L}}\otimes P_{n}(\mathbb{C}_p)\right)\times \left(\Sigma_{\mathscr{L}}\otimes P_{n}(\mathbb{C}_p)\right)\longrightarrow\mathbb{C}_p,$$ which we will denote by $\langle \cdot, \cdot \rangle _B$.
In this case, the duality between $$F_{\mathscr{L}}^{n} \times \Phi_{\mathscr{L}}^0/\Phi_{\mathscr{L}}^1\longrightarrow\mathbb{C}_p$$ defines the perfect pairing between $\mathcal{O}_\mathscr{L}(2-k,\mathcal{H}_F)$ and $ \Sigma_{\mathscr{L}}(k,F)$, which we will also refer to as the Breuil pairing and denote by $\langle \cdot, \cdot \rangle_B$.

\begin{proposition}\label{breuilPropertiesProp}
The Breuil pairing satisfies the following
\begin{enumerate}
\item If $\tau\in \mathcal{H}_F$ and $f\in \mathcal{O}_\mathscr{L}(2-k)$, then $$\left\langle f(z), \frac{(t-\tau)^{n}}{n!}\log_{\mathscr{L}}(t-\tau)\right\rangle _B  = f(\tau),$$ 
\item if $g(t)\in \Sigma(k)$, then $$\langle f(z),g(t)\rangle_B = \left\langle \frac{d^{n+1}}{dz^{n+1}}(f)(z), g(t)\right\rangle _M,$$
\item if $f(z) \in \mathcal{O}(2-k)$, then $$\langle f(z),g(t)\rangle _B = \left\langle f(z), (-1)^{n+1}\frac{d^{n+1}}{dt^{n+1}}(g)(t)\right\rangle _M.$$

\end{enumerate}

\end{proposition}

\subsection{The Schneider-Teitelbaum Lift}

Recall that for every oriented edge $e\in \mathcal{T}_F^1$, there is an associated oriented annulus, $\red^{-1}_F(e)$. Therefore, for every $f\in \mathcal{O}(2,\mathcal{H}_F)$ there is a well-defined annular residue $\mathrm{res}_e(f)$ at $e$, which is associated to the differential $f(z) dz$. 
\begin{definition}
For an oriented edge $e = (w_1,w_2)$, let $s(e) \defeq w_1,\; t(e) \defeq w_2$ denote the start and terminal vertices of $e$ and $\overline{e} \defeq (w_2, w_1)$ be the edge with the opposite orientation.

\end{definition} 
The residue theorem implies that for any vertex $v\in \mathcal{H}_F$ and rigid analytic function $f$,  \begin{equation}\label{harmonic}\sum_{s(e)=v}\textrm{res}_e(f) = 0, \end{equation} and for every edge $e$, \begin{equation}\textrm{res}_e(f) = -\textrm{res}_{\overline{e}}(f).\end{equation} In other words, every rigid analytic function $f(z)\in\mathcal{O}(2,\mathcal{H}_F)$ determines a residue functions on the edges $$c_f : \mathcal{T}_F^1\longrightarrow\mathbb{C}_p.$$ 

\begin{definition} \label{def: functions on tree}
	For any $\textrm{GL}_2(F)$-module $M$, we will denote by $C_F^0(M)$, the $\Gamma$-module of functions on the vertices of the tree valued in $M$: $$\mathcal{T}_F^0\longrightarrow M.$$
	Similarly for edges, we will denote by $C_F^1(M)$ the $\Gamma$-module of functions  $$c: \mathcal{T}_F^1 \longrightarrow M $$ with the extra requirement that $c(e) = -c(\overline{e})$. When the module $M$ is not specified then the functions take values in $\mathbb{C}_p$.
	Let $\nabla: C_F^1(M) \longrightarrow C_F^0(M)$ be the map taking a function $c\in C_F^1(M)$ to the function $$\nabla(c): v\mapsto \sum_{s(e)=v}c(e).$$
	The kernel of $\nabla$ is the space of harmonic functions, denoted by $C^1_{har}(M)$. 
\end{definition}

Since $\mathcal{T}_F$ is an infinite tree, the map $\nabla$ is surjective, and there is an exact sequence \begin{equation}\label{har}0\longrightarrow C^1_{har}(M) \longrightarrow C_F^1(M) \longrightarrow C_F^0(M) \longrightarrow0.\end{equation}

From (\ref{harmonic}) we can see that the function $c_f$ attached to a rigid analytic function is harmonic. In fact, the map $f\mapsto c_f$ fits in the exact sequence $$ 0\rightarrow\mathcal{O}(\mathcal{H}_F)\xrightarrow{d}\mathcal{O}(2,\mathcal{H}_F)\longrightarrow C^1_{har} \longrightarrow 0.$$
We will impose an integral structure on $C^1_{har}(P_n), C_F^1(P_n), C_F^0(P_n)$ as follows: For the standard vertex $v_0$, set $P_{n, v_0} \defeq P_n(\mathcal{O}_{\mathbb{C}_p})$. For other vertices, define for  $g\in \textrm{GL}_2(F)$  $$P_{n, g\cdot v_0} \defeq P_{n,v_0}\vert g^{-1}.$$
For every edge, $e = (v,v')$, let $$P_{n, e} \defeq P_{n,v} \cap P_{n,v'}.$$

On the spaces of functions let $$C_F^{0, \mathrm{int}}(P_n) \defeq \left\{d\in C_F^0(P_n)\; |\;  d(v)\in P_{n,v} \text{ for all } v\in \mathcal{T}_F^0 \right\},$$

$$C_F^{1, \mathrm{int}}(P_n) = \left\{c\in C_F^1(P_n) \; |\; c(e)\in P_{n,e} \text{ for all } e\in \mathcal{T}_F^1 \right\},$$
and
$$C_{F, har}^{1,\mathrm{int}}(P_n) = C^1_{F,har}(P_n) \cap C_F^{1, \mathrm{int}}(P_n).$$
Finally, the space of bounded harmonic functions is the base-change $$C_{F,har}^{1,b}(P_n) = C_{F, har}^{1,\mathrm{int}}(P_n) \otimes_{\mathcal{O}_{\mathbb{C}_p}}\mathbb{C}_p.$$

It will be a helpful to view harmonic functions on $\mathcal{T}_F$ as packaging values of measures on the boundary of the tree. 
\begin{definition}
For any $\textrm{GL}_2(F)$-module, $M$, let $C_F^{lc}(M)$ denote the space of locally constant functions on $\mathbb{P}^1(F)$ valued in $M$. Define a measure on $\mathbb{P}^1(F)$ valued in $M$ to be a functional $$C_F^{lc}(M^\vee) \longrightarrow \mathbb{C}_p.$$ The $\textrm{GL}_2(F)$-module of measures of mass zero is denoted by $\mathbb{M}(\mathbb{P}^1(F),M)$. A measure of total mass zero is one which vanishes on constant functions, the submodule of all such measures will be denoted by $\mathbb{M}_0(\mathbb{P}^1(F), M)$.

Given a harmonic function $c\in C^1_{F,har}(M)$, there is an associated measure of total mass zero defined by $$\mu(U_e) = c(e),\;\;\; \text{for } e\in \mathcal{T}_F^1.$$ This defines a measure since the opens sets $\{ U_e\}$ form a basis for the topology of $\mathbb{P}^1(F)$ and  the harmonicity ensures the additivity property of the measure and implies the total mass zero criterion.
\end{definition}

By Morita duality, we have identified $\mathcal{O}(k,\mathcal{H}_F)$ with the dual of $$\Sigma(k,F) = C(k,F)/P_n(\mathbb{C}_p).$$ The residue map fits into the diagram
$$\begin{tikzcd}
	\mathcal{O}(k,\mathcal{H}_F)\ar[rd,"\mathrm{res}"] \ar[r,"\sim" ] & \Sigma(k,F)\ar[d]^\vee\\
	& \left(C^{lc}/P_n(\mathbb{C}_p)\right)^\vee \cong C^1_{F, har}(P_n(\mathbb{C}_p)).
\end{tikzcd}$$
Where the vertical arrow is induced by the inclusion $C^{lc}$ into $C(k,F)$.

There is no $\mathrm{GL}_2(F)$-invariant section $C^1_{F,har}\longrightarrow \Sigma(2,F)^\vee$, this is because not all measures extend canonically to distributions on locally analytic functions. However, such a section exists if we restrict to bounded harmonic functions. 

\begin{theorem}[Vishik, Amice-Velu, Schneider-Teitelbaum]
There is a $\mathrm{GL}_2(F)$-equivariant section 
$$\textrm{ST}: C_{F,har}^{1,b}(P_n(\mathbb{C}_p)) \longrightarrow\mathcal{O}(k,\mathcal{H}_F).$$  We will refer to this map as the Schneider--Teitelbaum map.
\end{theorem}

\begin{proof}
	A bounded harmonic function, $\phi$, gives rise to a bounded measure $$\mu \in \mathbb{M}_0(\mathbb{P}^1(F),P_n).$$ The boundedness translates to bounds on integrals of the form $$\int_{U(e)}(t-a)^i\; d\mu(t), \;\; \mathrm{where } a\in U_e.$$
	Given a locally analytic function $g\in C(k,F)$, it can be approximated by a sequence of locally constant functions $g_i \in C_F^{lc}(P_n(\mathbb{C}_p))$. The bounds on the integrals of $\mu$ imply that $$\int_{\mathbb{P}^1(F)}g(t)\; d\mu(t) \defeq \lim_{i\to \infty} \int_{\mathbb{P}^1(F)}g_i(t)\; d\mu(t)$$ converges. In particular, $$ST(\phi)(z) \defeq \int_{\mathbb{P}^1(F)}\frac{1}{t-z}\; d\mu(t)$$ converges to a rigid analytic function in $z$, whose annular residues recover the function $\phi$. 
	
	For more details see \cite{DT},  \cite{MTT} and \cite{Orton}
\end{proof}

\begin{definition}
	The image $ST(C_{har}^{1,b}(P_n(\mathbb{C}_p)))$ is the space of bounded functions, which we will denote by $\mathcal{O}^b(k,\mathcal{H}_F)$. An alternative definition of $\mathcal{O}(k,\mathcal{H}_F)^b$ can be given as follows

Let $\lvert\cdot \rvert$ define a $\textrm{GL}_2(\mathcal{O}_F)$-invariant norm on $\mathcal{O}(k,\mathcal{H}_F)$ given by $$\lvert f \rvert \defeq \max_{z\in \red^{-1}(v_0)}\lvert f(z)\rvert .$$   We extend this to a $\textrm{GL}_2$-invariant norm by $$\|f\| \defeq \sup_{\gamma \in \textrm{GL}_2(F)} \lvert\left(f|_k\gamma\right)\rvert.$$ The space $\mathcal{O}(k,\mathcal{H}_F)^b$ is then precisely the space of all $f$ such that $||f||<\infty$.\\ 

\end{definition}

\begin{theorem}[\cite{dSBounded} Theorem 1.1]
	The Schneider-Teitelbaum map is an isomorphism $$ ST: C^{1,b}_{F,har}(P_n(\mathbb{C}_p)) \xrightarrow{\cong} \mathcal{O}^b(k,\mathcal{H}_F).$$

\end{theorem}

\subsection{Calculations}

We will record here some lemmas which will be useful later on. 

\begin{lemma}
	\label{explInteglemma}
	Let $E$ be any field, for any $\tau\in E$ and $P(T)\in E[T]$ a polynomial of degree $n$. Let $$d^{(n+1)}\left(P(z) \log(z-\tau)\right)$$ be the rational function defined by formally differentiating (n+1)-times the expression $P(z)\log(z-\tau)$. Then $$d^{(n+1)}\left(P(z) \log(z-\tau)\right) = \sum_{i=0}^n(-1)^i (n-i)! \frac{P^{(i)}(\tau)}{(z-\tau)^{n+1-i}}.$$
\end{lemma}

\begin{proof}
Expanding $P(z)$ in the basis $(z-\tau)^i$ and differentiating we obtain

$$d^{(n+1)}\left((P(z) \log(z-\tau)\right)  = \sum_{a = 0}^n \sum_{j = 0}^{n-a} {{n+1}\choose{a}} \frac{P^{(j+a)}(\tau)}{j!}(z-\tau)^j \cdot \frac{(-1)^{n-a}(n-a)!}{(z-\tau)^{n+1-a}}$$

Rearranging the sums with $i = a+j$ we obtain $$ = \sum_{i = 0}^n \left(\sum_{a =0 }^i (-1)^{n-a} {{n+1}\choose{a}}\frac{(n-a)!}{(i-a)!}  \right) \frac{P^{(i)}(\tau)}{(z-\tau)^{n+1-i}}$$ 

The equality $$\sum_{a=0}^i (-1)^{n-a}\binom{n+1}{a} \frac{(n-a)!}{(i-a)!} = (-1)^i (n-i)!$$ follows from Vandermonde's identity. 

\end{proof}

\begin{lemma} \label{STFin}
Let $\displaystyle{\mathcal{D} = \sum_i (\tau_i)\otimes P_i}$ be a degree $0$ divisor in $\mathrm{Div}_0(\mathcal{H}_p)\otimes P_n(\mathbb{C}_p)$. Then the measure $$\mu_{\mathcal{D}}(U) = \sum_{\tau_i\in U} P_i$$ is a bounded measure on $\mathbb{P}^1(F)$, whose Schneider-Tetielbaum lift is described as $$ST(\mu_\mathcal{D})(z) = \frac{1}{n!}\sum_i d^{n+1}(P_i(z)\log(z-\tau_i)).$$
\end{lemma}

\begin{proof}
	Let $\{e_i\}_i$ be any set of edges on $\mathcal{T}_F$ such that $U(e_i)$ are all pairwise disjoint and $\tau_i \in U(e_i)$.
	\begin{align*} 
		ST(\mu_\mathcal{D})(z) &= \int_{\mathbb{P}^1(F)} \frac{1}{t-z}\; d\mu_T(t) \\ 
		&= \sum_i \int_{U(e_i)}\frac{1}{t-z}\; d\mu_\mathcal{D}(t) \\
		&= \sum_i \sum_{j = 0}^{\infty}\frac{\int_{U(e_i)}(t-\tau_i)^j \; d\mu_\mathcal{D}(t)}{(z-\tau_i)^{j+1}}\\
		& = \sum_i \sum_{j=0}^n (-1)^j \frac{j!}{n!} \frac{P_i^{n-j}(\tau_i)}{(z-\tau_i)^{j+1}}\\
		& = \frac{1}{n!}\sum_i d^{n+1}\left(P_i(z)\log(z-\tau_i)\right).
	\end{align*}

\end{proof}

\subsection{The Ihara group}

For the remainder of this paper we will be primarily concerned with the cases where $F = \mathbb{Q}_p$ or $F=K\defeq $ the unramified quadratic extension of $\mathbb{Q}_p$. To simplify notation we will write $\mathcal{H}_p$ for $\mathcal{H}_{\mathbb{Q}_p}$ and $\mathcal{T}_p$ for $\mathcal{T}_{\mathbb{Q}_p}$ and whenever the field is unspecified in $C^1_{har}, C^0, C^1$ it is assumed to be $\mathbb{Q}_p$.

Let $\Gamma \defeq \textrm{SL}_2(\mathbb{Z}[1/p])$ be the Ihara group. The action of $\Gamma$ on vertices of $\mathcal{T}_p$ has two orbits; two vertices are in the same orbit if and only if the distance between them is even. The stabilizer of the vertex $v_0$ is $\textrm{SL}_2(\mathbb{Z})$ and the stabilizer of any other vertex is a conjugate of $\textrm{SL}_2(\mathbb{Z})$ in $\textrm{GL}_2(\mathbb{Z}[1/p])$. In particular, the stabilizer of $v_1$ is $$P^{-1}\cdot\textrm{SL}_2(\mathbb{Z})\cdot P, $$
where 
$$ P = \begin{pmatrix} p& 0 \\ 0 & 1\end{pmatrix}.$$
The action of $\Gamma$ on non-oriented edges is transitive, and the stabilizer of the standard edge is $$ \Gamma_0(p) = P^{-1}\textrm{SL}_2(\mathbb{Z})P\cap \textrm{SL}_2(\mathbb{Z}).$$ 
There is a natural inclusion of trees $$i: \mathcal{T}_p \xhookrightarrow{}\mathcal{T}_K,$$ which is given by base-changing $\mathbb{Z}_p$ lattices to $\mathcal{O}_K$ lattices. The quotient $\Gamma \backslash\mathcal{T}_K$ is not finite, but consists of two infinite rays beginning at the images of $v_0$ and $v_1$, respectively, as well as the edge $e_0$. 

For every vertex $v\in \mathcal{T}_p$, let $\mathcal{P}(v,K)$ be the set of infinite non-backtracking paths $$(v, w_1, w_2 ,\cdots)$$ on $\mathcal{T}_K$ which begin at $v$ and such that $w_i$ belongs to $\mathcal{T}_K\backslash \mathcal{T}_p$ for every $i \geq 1$.

\begin{lemma}
	\label{merextLem}
	Let $f\in \mathcal{O}(k,\mathcal{H}_K)^b $ be a bounded rigid analytic function of weight $k$ and $\phi_f$ its associated harmonic function on $\mathcal{T}_K^1$. If $f$ satisfies that for every vertex $v \in \mathcal{T}_p^1$ the function $\phi_f$ eventually stabilizes to $0$ on all but finitely many paths in $\mathcal{P}(v,K)$. Then $f$ has a unique extension to a rigid meromorphic function on $\mathcal{H}_p$.
\end{lemma}

\begin{proof}
	On the affinoid $\mathcal{H}_K^{\leq l}$ pick the representatives $\mathcal{R}_l$ such that the representatives of elements in $\mathbb{P}^1\left(\mathbb{Z}_p/ p^{l}\mathbb{Z}_p\right)$ all belong to $\mathbb{P}^1(\mathbb{Q}_p)$. Then we can write $$f =\sum_{i=0}^\infty b_{\infty,i}z^i +  \sum_{\substack{a\in \mathcal{R}_l \\ a\in \mathbb{P}^1(\mathbb{Q}_p)\backslash \{\infty\}}} \sum_{i=1}^\infty \frac{b_{a_i}}{(z-a)^i} + Y_l,$$
	where $Y_l$ is a rational function which can be described as follows: Let $\tau_1,\cdots, \tau_c$ be the points in $\mathbb{P}^1(K) \backslash \mathbb{P}^1(\mathbb{Q}_p)$ corresponding to the finitely many paths beginning at a vertex in $\mathcal{T}_p^{\leq l}$ and whose edges belong to $\mathcal{T}_K\backslash \mathcal{T}_p$, and let $P_1,\cdots, P_t$ be the polynomials to which these paths stabilize Then by Lemma \ref{explInteglemma}, $Y_l$ can be written as $$Y_l =  \sum_{j=1}^t d^{(n+1)}\left(P_j(z)\log(z-\tau_j)\right).$$

	The sum of power series $$\sum_{i=0}^\infty b_{\infty,i}z^i +  \sum_{\substack{a\in \mathcal{R}_l \\ a\in \mathbb{P}^1(\mathbb{Q}_p)\backslash \{\infty\}}} \sum_{i=1}^\infty \frac{b_{a_i}}{(z-a)^i}$$ in fact converges on $\mathcal{H}_p^{\leq l}$ and the rational function $Y_l$ is meromorphic on $\mathcal{H}_p$; their sum is therefore a rigid meromorphic function on $\mathcal{H}_p^{\leq l}$ which extends $f$. 
\end{proof}

\begin{definition}
	We will denote the space of rigid meromorphic functions on $\mathcal{H}_p$ whose poles all lie in $\mathbb{P}^1(K)$ and whose poles are of order at most $k$ by $\mathcal{M}(k;K)$. By Lemma \ref{merextLem}, this space is identified with the space of rigid analytic functions on $\mathcal{H}_K$, that satisfy the hypothesis of the lemma. We will also denote by $\mathcal{M}_{\mathscr{L}}(2-k ; K)$ the set of all functions in $\mathcal{O}_{\mathscr{L}}(2-k, \mathcal{H}_K)$ whose $(n+1)$-th derivative lies in $\mathcal{M}(k;K)$. 
\end{definition}

\section{Higher Weight Rigid Cocycles}

In this section, we will discuss rigid meromorphic cocycles for the group $$\Gamma \defeq \mathrm{SL}_2(\mathbb{Z}[1/p])$$ acting on $\mathcal{H}_p$ and $\mathcal{H}_K$. First, we will recall the definition of the multiplicative rigid meromorphic cocycles of \cite{DV1}, and describe their associated boundary measures on $\mathbb{P}^1(K)$. Then, we will generalize this construction to higher weight and use it to give a classification of rigid meromorphic cocycles.

\subsection{Definitions of cocycles and modular symbols}

\begin{definition}\label{def: modular symbols}
For any (right) $\Gamma$-module $M$, a $1$-cocycle is an element of the group $Z^1(\Gamma,M)$ of functions $J:\Gamma \longrightarrow  M$ that satisfy $$J(\gamma_1 \gamma_2) = J(\gamma_1)\cdot \gamma_2 + J(\gamma_2),\;\; \text{for all } \gamma_1, \gamma_2\in \Gamma.$$ The cohomology group $H^1(\Gamma, M)$ is the quotient of $Z^1(\Gamma, M)$ with the subgroup generated by co-boundaries; those are functions of the form $$\gamma \mapsto f\cdot \gamma - f,$$ for some fixed $f\in M$.

\end{definition}

\begin{definition}
Let $\mathbb{D}_0$ be the group of degree-zero divisors in $\mathbb{P}^1(\mathbb{Q})$, which is defined by the exact sequence \begin{equation}\label{equationn: deg 0} 0\longrightarrow  \mathbb{D}_0 \longrightarrow  \mathbb{Z}[\mathbb{P}^1(\mathbb{Q})] \xrightarrow{\mathrm{deg}} \mathbb{Z} \longrightarrow  0.\end{equation} If $G$ is any group which acts transitively on $\mathbb{P}^1(\mathbb{Q})$, for any $G$-module, $M$, an $M$-valued modular symbol is a homomorphism of abelian groups $$\mathbb{D}_0 \longrightarrow  M.$$ We will denote by $\mathrm{MS}(M)$ the group of $M$-valued modular symbols. As $r,s$ range over $\mathbb{P}^1(\mathbb{Q})$, divisors of the form $(r)-(s)$ span the group $\mathbb{D}_0$. We will denote such divisors by $\{r,s\}$. The right action of $G$ on $\mathrm{MS}(M)$ is described by $$(m\cdot \gamma)\{r,s\} = m\{\gamma r, \gamma s\}\cdot \gamma.$$ Let $\mathrm{MS}^G(M)$ be the subgroup of $\mathrm{MS}(M)$ of modular symbols fixed by $G$.
\end{definition}

The short exact sequence (\ref{equationn: deg 0}) yields the exact sequence \begin{equation}\label{equation: deg0 dual}0 \longrightarrow  M \cong\mathrm{Hom}(\mathbb{Z},M) \longrightarrow  \mathrm{Hom}\left(\mathbb{Z}[\mathbb{P}^1(\mathbb{Q})],M\right) \longrightarrow  \mathrm{Hom}(\mathbb{D}_0,M) \longrightarrow  0.\end{equation} By taking $G$-cohomology of (\ref{equation: deg0 dual}), the first $\delta$-function maps $$\delta: \left(\mathrm{Hom}(\mathbb{D}_0,M)\right)^G \longrightarrow H^1(\Gamma, M).$$ This map can be described as follows: Chose an arbitrary $r\in \mathbb{P}^1(\mathbb{Q})$. Given a $\Gamma$-invariant modular symbol $m$, then define $$J_m(\gamma) \defeq m\{r,\gamma r\}, \quad \text{for all }\gamma \in G.$$ The cohomology class of $J_m$ is independent of the choice of $r$ and is equal to $\delta(m)$.

Let $M^\dagger$ denote the set of elements $h\in M$ that satisfy the two-term relation $$h\cdot S + h = 0,$$ and the three-term relation $$h\cdot U^2 + h\cdot U + h = 0,$$
where $S = \begin{pmatrix} 0 & -1 \\ 1 & 0 \end{pmatrix}$, and $U = \begin{pmatrix}0 & 1 \\ -1 & 1 \end{pmatrix}$. 
\begin{proposition}\label{modularPeriodProp} For any $\mathrm{SL}_2(\mathbb{Z})$-module $M$, the map $$m \mapsto m\{0,\infty\}$$ is a bijection between $\mathrm{MS}^{\mathrm{SL}_2(\mathbb{Z})}(M)$ and $M^\dagger$.
\end{proposition}

\begin{proof}
A proof of this statement can be found in \cite{DV1} Proposition 1.4. This statement holds because $\mathrm{SL}_2(\mathbb{Z})$ acts transitively on the set of pairs $\{\frac{a}{c}, \frac{b}{d}\} \in \mathbb{D}_0$ such that $ad-bc = 1$ (with the convention $\infty = \frac{1}{0}$), and the set of all such pairs generates $\mathbb{D}_0$ (by the Euclidean algorithm). 
\end{proof}

We will now consider modular symbols and group cohomology valued in spaces of rigid analytic functions. It is crucial here that the rigid analytic functions which arise as values of cocycles or modular symbols are necessarily bounded (Proposition \ref{proposition: bounded cocycles}). This holds due to the $\Gamma$-invariance condition on cocycles and modular symbols, along with the fact that $\Gamma$ is finitely generated. 

\begin{proposition}\label{proposition: bounded cocycles}
\begin{enumerate}
\item The homorphisms induced by the natural inclusion $$H^1(\Gamma, \mathcal{O}(k)^b) \longrightarrow  H^1(\Gamma, \mathcal{O}(k, \mathcal{H}_p))$$ and $$MS^\Gamma(\mathcal{O}(k)^b)\longrightarrow \mathrm{MS}^\Gamma(\mathcal{O}(k, \mathcal{H}_p))$$ are isomorphisms. 
\item The homomorphisms induced by the natural inclusion $$H^1(\Gamma, C^{1,b}_{har}(P_n)) \longrightarrow  H^1(\Gamma, C^1_{har}(P_n))$$ and $$MS^\Gamma (C^{1,b}_{har}(P_n))\longrightarrow \mathrm{MS}^\Gamma(C^1_{har}(P_n))$$ are isomorphisms. 
\end{enumerate}

\end{proposition}

\begin{proof}
    Both of these statements follow from \cite[Proposition 3.7]{RS}.
\end{proof}

\begin{corollary}\label{boundedCor}
The residue map induces isomorphisms $$H^1(\Gamma, \mathcal{O}(k, \mathcal{H}_p)) \longrightarrow  H^1(\Gamma, C^1_{har}(P_n))$$ 
and $$MS^\Gamma(\mathcal{O}(k, \mathcal{H}_p))\longrightarrow MS^\Gamma(C^1_{har}(P_n)).$$
\end{corollary}

Corollary \ref{boundedCor} reduces the construction of rigid analytic modular symbols and cohomology classes to constructing harmonic functions on the Bruhat--Tits tree, where it is possible to exploit the relatively simple combinatorics of the tree and the action of $\Gamma$ on it.

After taking $\Gamma$-cohomology, the short exact sequence (\ref{har}) yields
$$ \longrightarrow  H^0(\Gamma, C_p^0(P_n)) \longrightarrow  H^1(\Gamma, C^1_{p,har}(P_n))\longrightarrow  H^1(\Gamma, C_p^1(P_n)) \longrightarrow  H^0(\Gamma, C_p^1(P_n))\longrightarrow   .$$

Since the action of $\Gamma$ on non-oriented edges is transitive and the stabilizer of an edge is conjugate to $\Gamma_0(p)$, there is an isomorphism $$ C_p^1(P_n) \cong \mathrm{Ind}_{\Gamma_0(p)}^{\Gamma}P_n.$$ Therefore, Shapiro's lemma provides us with an isomorphism \begin{equation}\label{equation: Shaprio iso edges}H^1(\Gamma, C_p^1(P_n))\cong H^1(\Gamma_0(p), P_n).\end{equation} Similarly, $\Gamma$ has two orbits of vertices with stabilizers conjugate to $\mathrm{SL}_2(\mathbb{Z})$, so \begin{equation} \label{equation: Shaprio iso vertices} H^1(\Gamma, C_p^0(P_n))\cong H^1(\mathrm{SL}_2(\mathbb{Z}),P_n)^2.\end{equation}

Notice that for $k>2$ the group $H^0(\Gamma, C_p^0(P_n))$ is trivial. It follows from (\ref{equation: Shaprio iso edges}) and (\ref{equation: Shaprio iso vertices}) that we can identify $$H^1(\Gamma, C^1_{har}(P_n))\cong \mathrm{Ker}\left[H^1(\Gamma_0(p), P_n)\longrightarrow  H^1(\mathrm{SL}_2(\mathbb{Z}),P_n)^2\right],$$ where the maps $H^1(\Gamma_0(p),P_n) \rightarrow H^1(\mathrm{SL}_2(\mathbb{Z}),P_n)$ are the two trace maps.

\begin{theorem}[Shimura, Eichler-Shimura] 
For all $k\geq 2$
\begin{enumerate}
\item Let $f$ be a $p$-new normalized eigenform of weight $k$ for $\Gamma_0(p)$, then there exists transcendental periods $\Omega_f^{\pm}$ such that for every $\gamma \in \Gamma_0(p)$, the cocycle $$p_f^{\pm}(\gamma) \defeq \frac{1}{\Omega_f^{\pm}}\int_{\eta}^{\gamma \eta}\frac{\omega_f \pm \omega_f^{\pm}}{2}$$ takes algebraic values, where $\eta$ is a fixed choice of base-point and $\omega_f$ is the differential associated to $f$ on the modular curve $X(\Gamma_0(p))$, 
\item $p_f^{\pm}\in Z^1(\Gamma_0(p), P_n)$ is a cocycle whose cohomology class lies in $$\mathrm{Ker}\left(H^1(\Gamma_0(p), P_n)\longrightarrow  H^1(\mathrm{SL}_2(\mathbb{Z}),P_n)^2\right),$$
\item If $k>2$, as $f$ ranges over all $p$-new normalized eigenforms, the cocycles $p_f^\pm$ form a basis of $$\mathrm{Ker}\left(H^1(\Gamma_0(p), P_n)\longrightarrow  H^1(\mathrm{SL}_2(\mathbb{Z}),P_n)^2\right).$$
\end{enumerate}
\end{theorem}

\begin{proof}
    It is well-known (e.g. see \cite{AshStevens} Theorem 2.3) that for any $N$, there is an isomorphism of Hecke-modules $$H^1(\Gamma_0(N), P_n) \cong S_k(\Gamma_0(N)) \oplus \overline{S_k(\Gamma_0(N))} \oplus \mathcal{E}_k(\Gamma_0(N)),$$
    where $\mathcal{E}_k(\Gamma_0(N))$ is the space of Eisenstein series for the congruence subgroup $\Gamma_0(N)$. ($2$) then follows from the condition that $f$ is $p$-new. ($3$) holds because the Eisenstein series of $\Gamma_0(p)$ are not in the kernel of the two trace maps $$M_{k}(\Gamma_0(p)) \longrightarrow M_k(\mathrm{SL}_2(\mathbb{Z})).$$
\end{proof}

\begin{corollary}\label{tophoriso}
For $k>2$, the map $$\mathrm{MS}^{\Gamma}(\mathcal{O}(k,\mathcal{H}_p)) \longrightarrow H^1(\Gamma,\mathcal{O}(k,\mathcal{H}_p))$$ is an isomorphism.
\end{corollary}

\begin{proof}
As Hecke-modules, both $\mathrm{MS}^{\Gamma}(\mathcal{O}(k,\mathcal{H}_p))$ and $H^1(\Gamma,\mathcal{O}(k,\mathcal{H}_p))$ are isomorphic to  $H^1(\Gamma_0(N), P_n) \cong S_k(\Gamma_0(N)) \oplus \overline{S_k(\Gamma_0(N))}$, and the map between them is Hecke-equivariant and is therefore an isomorphism. 
\end{proof}

\begin{lemma}
If $\Gamma_0$ is any congruence subgroup, then \begin{enumerate}
\item $\displaystyle{H^2(\Gamma_0,P_n) = 0}$,
\item $\displaystyle{H^1(\Gamma_0, \mathrm{MS}(P_n))} = 0$. 
\end{enumerate}
\end{lemma}

\begin{proof}
\begin{enumerate}
\item follows from \cite{Hida} p.162 Proposition 1. 
\item follows from the same proposition after identifying $H^1(\Gamma, \mathrm{MS}(P_n))$ with Hida's $H^2_P(\Gamma, P_n)$ for example by combining \cite{AshStevens} Proposition 4.2 and p. 363 Proposition 5 of \cite{Hida}. 
\end{enumerate}
\end{proof}

\begin{lemma}\label{HidaH2van}
The maps $$H^1(\Gamma, \mathcal{O}(k,\mathcal{H}_p) \rightarrow H^2(\Gamma, P_n)$$ and $$H^0(\Gamma,\mathrm{MS}(\mathcal{O}(k,\mathcal{H}_p)))\rightarrow H^1(\Gamma, \mathrm{MS}(P_n ))$$ are isomorphisms. 
\end{lemma}
\begin{proof}
Consider the exact sequence:

$$0\rightarrow P_n\rightarrow C^0(P_n) \rightarrow C^1(P_n)\rightarrow 0.$$ Take the associated long exact sequences $$\cdots \rightarrow H^1(\Gamma,C^0(P_n)) \rightarrow H^1(\Gamma, C^1(P_n))\rightarrow H^2(\Gamma, P_n)\rightarrow H^2(\Gamma, C^0(P_n)) \rightarrow \cdots $$ 

Shapiro's lemma and Lemma \ref{HidaH2van} imply that $H^2(\Gamma, C^0(P_n))$ is trivial. Further, the inclusion of $H^1(\Gamma, C^1_{\text{har}}(P_n))$ into $H^1(\Gamma, C^1)$ is complimentary to the image of $H^1(\Gamma, C^0(P_n))$, and therefore $H^1(\Gamma, C^1_{har}(P_n))$ is isomorphic to $H^2(\Gamma, P_n)$.

The same argument proves the second assertion for modular symbols. 
\end{proof}

\begin{proposition}\label{obstructionIso}
The map $$H^1(\Gamma, \mathrm{MS}(P_n))\longrightarrow H^2(\Gamma,P_n)$$ is an isomorphism. 
\end{proposition}

\begin{proof}
This follows from the commutative diagram
\begin{center}
\begin{tikzcd}
    H^0(\Gamma, \mathrm{MS}(\mathcal{O}(k,\mathcal{H}_p)))\ar[r] \ar[d] &H^1(\Gamma, \mathcal{O}(k,\mathcal{H}_p)) \ar[d]\\H^1(\Gamma, \mathrm{MS}(P_n)) \ar[r]&H^2(\Gamma,P_n).
\end{tikzcd}
\end{center}
Lemma \ref{HidaH2van} shows that the vertical arrows are isomorphisms and Corollary \ref{tophoriso} shows that the top horizontal arrow is an isomorphism.
\end{proof}
\subsection{RM-points and Binary Quadratic Forms}

Fix once and for all embeddings of $\overline{\mathbb{Q}}$ into $\mathbb{C}$ and $\mathbb{C}_p$.
\begin{definition}
An RM-point is a point $\tau \in \mathcal{H}_p$, such that $\mathbb{Q}(\tau)$ is a real-quadratic field. This necessarily implies that $p$ is either inert or ramified in $\mathbb{Q}(\tau)$. Throughout this paper we will restrict our attention to inert RM points. The same methods apply to ramified RM-points and will be treated in the author's thesis.
\end{definition}

To each inert RM-point, $\tau$, there is a unique binary quadratic form $$Q_\tau(x,y) = ax^2 + bxy + cy^2,$$ with $a,b,c\in \mathbb{Z}$ co-prime, such that $Q(\tau,1) = 0$ and $\tau$ is the ``positive root''; i.e. $$\tau = \frac{-b + \sqrt{b^2-4ac}}{2a}.$$
\begin{remark}
Here the sign of the roots is determined via the fixed embedding of $\overline{\mathbb{Q}}$ into $\mathbb{C}$.
\end{remark}
The discriminant of $\tau$ is $$D_\tau = \text{Disc}(Q_\tau) = b^2-4ac.$$ The discriminant factorizes as  $D_\tau = D_0 p^{2l}$, where $D_0$ is co-prime to $p$ and $l$ is the level of $\mathrm{red}_{p}(\tau)$. 

We denote by $(\tau,\overline{\tau})$ the geodesic in the complex upper half-plane $\mathcal{H}$ whose endpoints are $\tau$ and $\overline{\tau}$. Note that this geodesic is fixed by the action of the stabilizer of $\tau$ in $\Gamma$. The stabilizer subgroup of $\tau$ in $\Gamma$ is isomorphic to the group of units in the order $\mathcal{O}_{\tau} \subset \mathbb{Q}(\tau)$. There is a unique stabilizer, $\gamma_\tau$, which corresponds to a positive generator of $\mathcal{O}_\tau^\times$ (equivalently, $c\tau + d>0$), and such that the iterated action of the matrix $\gamma_\tau$ on $\mathbb{C}$ has $\tau$ as a stable fixed point.

For two oriented geodesics $g_1$ and $g_2$ on the complex upper half-plane, $\mathcal{H}$, we will denote by $g_1\cdot g_2$ their signed topological intersection number, as defined in \cite{DV3} p.2. For a pair of cusps $r,s\in\mathbb{P}^1(\mathbb{Q})$, we will denote by $(r,s)$ the oriented geodesic in the complex upper half-plane whose endpoints are $r$ and $s$.

\begin{definition}\label{def: orbits of rm points}
    For a fixed pair $(r,s)\in \mathbb{P}^1(\mathbb{Q})^2$, define the set $$\Sigma_{\tau}\{r,s\} = \left\{w\in \Gamma \cdot \tau \mid (w,\overline{w})\cdot (r,s)\neq 0 \right\},$$ and for every $l$, the set $\Sigma_\tau^{\leq l} \{r,s\} $ is the subset of $\Sigma_\tau\{r,s\}$ of RM-points whose reduction to $\mathcal{T}_{p}$ is of level at most $l$. 
\end{definition}

\begin{lemma}[\cite{DV1} p.41]\label{finiteRMLemma}
For each $\{r,s\}$, the set $\Sigma^{\leq l}_\tau\{r,s\}$ is finite.

\end{lemma}

\begin{proof}
Finiteness of $\Sigma_{\tau}^{\leq l}\{ 0, \infty \}$ follows from the fact that $(\tau, \overline{\tau})\cdot (0, \infty)$ is non-zero if and only if $ac <0$ where $Q_\tau(x,y) = ax^2+bxy + cy^2$, and there are only finitely many choices of such $a,b,c$ for a fixed discriminant. For any pair $(r,s)$, the set $\Sigma_{\tau}^{\leq l}\{r,s\}$ can be written as a finite combination of translates of $\Sigma_\tau^{\leq l}\{r,s\}$, and is therefore finite.
\end{proof}

\begin{definition}[RM-divisors]\label{def: RM-divisors}
    An RM-divisor $\mathcal{D}$ is a finite formal sum of $\Gamma$-orbits of real quadratic points on $\mathcal{H}_p$, we will write such an RM-divisor as $$\sum_i m_i \cdot [\tau_i].$$
    Denote by $\mathcal{RM}$ the group of $RM$-divisors and $\mathcal{RM}^{in}$ the subgroup of inert RM-divisors; i.e. those where $p$ is inert in the field $\mathbb{Q}(\tau_i)$.
    We will denote by $\mathrm{Div}^\dagger(\mathcal{H}_p)$ the set of formal (possibly infinite) sums of points on $\mathcal{H}_p$ of the form $$D = \sum_{i\in I}a_i (z_i),$$ with $a_i\in \mathbb{Z}$ and $z_i \in \mathcal{H}_p$ such that the support of $D$  intersects each affinoid $\mathcal{H}_p^{\leq l}$ in a finite set, and the points $z_i$ are inert.
    
\end{definition}

For every inert RM-divisor, we can associate a formal divisor $$\mathrm{Div}_{k,\mathcal{D}} \in\mathrm{MS}^\Gamma(\mathrm{Div}^\dagger(\mathcal{H}_p)\otimes P_n(\mathbb{C}_p))$$ by $$\mathrm{Div}_{k, \mathcal{D}}\{r,s\} = \sum_i \sum_{w\in [\tau_i]}m_i \cdot (w)\otimes \left[(w,\overline{w})\cdot (r,s)\right] \frac{(T-w)^{n/2}(T-\overline{w})^{n/2}}{(w-\overline{w})^{n/2}} .$$
For every vertex, $v\in \mathcal{T}_p^0$, the intersection of a divisor with $\mathrm{red}^{-1}(v)$ is a finite divisor. The collection of the degrees of all such finite divisors can be assembled into a map $$\mathrm{deg}: \mathrm{Div}^\dagger(\mathcal{H}_p)\otimes P_n(\mathbb{C}_p) \longrightarrow  C_{p}^0(P_n(\mathbb{C}_p)).$$ The image of $\mathrm{Div}_{k,\mathcal{D}}$ under the degree map is  $$\mathrm{Deg}_{k,\mathcal{D}}\{r,s\}(v) \defeq \sum_i\sum_{\substack{w\in [\tau_i] \\ \mathrm{red}_{p}(w) = v }} m_i \left[(w,\overline{w})\cdot (r,s)\right] \frac{(T-w)^{n/2}(T-\overline{w})^{n/2}}{(w-\overline{w})^{n/2}}.$$

\begin{proposition}\label{ModularSymbolproposition} The symbols $\mathrm{Div}_{k,\mathcal{D}}$ and $ \mathrm{Deg}_{k,\mathcal{D}}$ are $\Gamma$-invariant modular symbols. 
\end{proposition}

\begin{proof}
    It suffices to prove this proposition for $\mathrm{Div}_{k,\mathcal{D}}$ since $\mathrm{Deg}_{k,\mathcal{D}}$ is its image under the degree map, which is $\Gamma$-equivariant. 
    The fact that $\mathrm{Div}_{k,\mathcal{D}}$ defines a modular symbol follows immediately from the observation that the map $$\{r,s\} \mapsto (w,\overline{w})\cdot(r,s)$$ is a modular symbol. 

    To show the $\Gamma$-invariance, note that 
     $$\frac{(T-w)^{n/2}(T-\overline{w})^{n/2}}{(w-\overline{w})^{n/2}}\mid \gamma = \frac{(T-\gamma^{-1}w)^{n/2}(T-\gamma^{-1}\overline{w})^{n/2}}{(\gamma^{-1}w-\gamma^{-1}\overline{w})^{n/2}}.$$
    So for $\gamma\in \Gamma$, \begin{align*}&\mathrm{Div}_{k,\mathcal{D}}\{r,s\}\mid \gamma = \sum_i \sum_{w\in \Gamma\cdot \tau_i }m_i \cdot (\gamma^{-1}w)\otimes \left[(w,\overline{w})\cdot (r,s)\right] \frac{(T-\gamma^{-1}w)^{n/2}(T-\gamma^{-1}\overline{w})^{n/2}}{(\gamma^{-1}w-\gamma^{-1}\overline{w})^{n/2}} \\
     &=\sum_i\sum_{w\in \Gamma\cdot \tau_i } m_i \cdot (\gamma^{-1}w)\otimes \left[(\gamma^{-1}w,\gamma^{-1}\overline{w})\cdot (\gamma^{-1}r,\gamma^{-1}s)\right] \frac{(T-\gamma^{-1}w)^{n/2}(T-\gamma^{-1}\overline{w})^{n/2}}{(\gamma^{-1}w-\gamma^{-1}\overline{w})^{n/2}}\\
     &= \sum_i\sum_{w\in \Gamma\cdot \tau_i }m_i\cdot (w)\otimes \left[(w,\overline{w})\cdot (\gamma^{-1}r,\gamma^{-1}s)\right] \frac{(T-w)^{n/2}(T-\overline{w})^{n/2}}{(w-\overline{w})^{n/2}}\\
     &= \mathrm{Div}_{k,\mathcal{D}}\{\gamma^{-1}r,\gamma^{-1}s\}.\end{align*}
\end{proof}

\begin{lemma}\label{wt2degtriv}
    For $k=2$, the modular symbol $\mathrm{Deg}_{2,\tau}$ is trivial. 
    \end{lemma}
    
    \begin{proof}
    Details of the proof can be found in \cite{DV1}, Lemma 1.16. It is closely related to the fact that the space of cusp forms of level $1$ and weight $2$ is trivial. 
    \end{proof}
\subsection{Weight $2$ Rigid Meromorphic Cocycles}

For this section, we will denote by $\mathcal{M}$ the fraction field of $\mathcal{O}(\mathcal{H}_p)$ and $\mathcal{M}^\times$ its multiplicative group. 
\begin{definition}
A rigid meromorphic theta symbol is a modular symbol in $$\mathrm{MS}^\Gamma(\mathcal{M}^\times/\mathbb{C}_p^\times).$$

\end{definition}
Recall the construction of rigid meromorphic theta symbols in \cite{DV1}.
First define for $w\in \mathcal{H}_p$, the choice of rational function with a pole at $w$:  \begin{equation} \label{twFun}t_w(z)  = \begin{cases}z-w & \text{if }|w|\leq 1,\\ \frac{z}{w} - 1& \text{if }|w|>1. \end{cases}\end{equation}

\begin{theorem}[Darmon-Vonk, \cite{DV1}]
For $\tau$ an RM-point, The symbol defined by $$J^\times_{\tau}\{r,s\}(z) = \prod_{w\in \Gamma \cdot \tau} t_w(z)^{(r,s)\cdot (w, \overline{w})}$$ is a modular symbol in $MS^\Gamma\left( \mathcal{M}^\times/ \mathbb{C}_p^\times \right)$ whose logarithmic derivative is $$J_{2,\tau}\{r,s\}(z) = \sum_{w\in\Gamma\cdot \tau} (r,s)\cdot (w, \overline{w})\frac{1}{z-w}.$$

The set of rigid analytic modular symbols $\mathrm{MS}^\Gamma(\mathcal{O}(2,\mathcal{H}_p))$ along with the modular symbols $\{J_{2,\tau}\}_{\tau}$, where $\tau$ ranges over all RM-points span the vector space $\mathrm{MS}^\Gamma(\mathcal{M}(2,K)).$
\end{theorem}

\begin{lemma}
Let $\tau$ be an inert RM-point and $r,s \in \mathbb{P}^1(\mathbb{Q})$. For an edge $e\in \mathcal{T}_K$ such that $U_e\cap \mathbb{P}^1(\mathbb{Q}_p) =\emptyset$, then the intersection $$U_e\cap \Sigma_\tau\{r,s\}$$ is finite.
\end{lemma}

\begin{proof}

    This follows from Lemma \ref{finiteRMLemma}, since as a subset of $\mathcal{H}_p$, the set $U_e$ reduces to a single vertex: $$\mathrm{red}_p(U_e) = v,$$ for some vertex $v\in \mathcal{T}_p^0$.
\end{proof}

\begin{proposition}
Let $\tau$ be an inert RM-point. When restricted to $\mathcal{H}_K$, the function $J_{2,\tau}\{r,s\}\vert_{\mathcal{H}_K}$ is rigid analytic and its associated harmonic function, $$\Phi_{2,\tau}\in \mathrm{MS}^\Gamma(C_{K,har}^1(P_n)),$$ can be described as follows: 

\begin{itemize}
\item for $e\in \mathcal{T}_{p}$, $$\Phi_{2,\tau}\{r,s\}(e) = 0,$$
\item for an edge $e\in \mathcal{T}_K \backslash \mathcal{T}_{p}$ which is pointing away from the standard vertex   $$\Phi_{2,\tau}\{r,s\}(e) =  \sum_{w\in \Sigma_\tau\{r,s\}\cap U_e}\left[(w,\overline{w})\cdot (r,s)\right].$$

\end{itemize}

\end{proposition}

\begin{proof}

Define the rational function on $\mathcal{H}_K$ given by  $$J^{\leq d}_{2,\tau}\{r,s\} = \sum_{w\in \Sigma^{\leq d}_\tau \{r,s\}}(r,s)\cdot(w,\overline{w}) \frac{1}{z-w}.$$ Let $\Phi_{2,\tau}^{\leq{d}}\{r,s\}$ be the residue of the function $J_{2,\tau}^{\leq d} $.

Given an oriented edge $e\in \mathcal{T}_K^1$ of level $l$ and which is pointing away from $v_0$. For $d > l$, the only elements of the sum defining $J_{2,\tau}^{\leq d}$ which contribute to the residue at $e$ are the RM-points $$w \in \Sigma_{\tau}^{\leq d}\{r,s\}\cap U_e.$$ Lemma \ref{wt2degtriv} implies that if $e\in \mathcal{T}_{p}^1$, then the sum of those residues contributing at $e$ is $0$, and so $\Phi^{\leq d}_{2,\tau}\{r,s\}(e) = 0$. If instead $U_e \cap \mathbb{P}^1(\mathbb{Q}_p) = \emptyset$, then by Lemma \ref{finiteRMLemma} the set $\Sigma_{\tau}\{r,s\}\cap U_e$ is finite and the residue is a finite sum
$$\Phi_{2,\tau}^{\leq d}\{r,s\} = \sum_{w\in \Sigma^{\leq d}_\tau\{r,s\}\cap U_e}\left[(w,\overline{w})\cdot (r,s)\right].$$ 

 The sequence of functions $$\left(\Phi_{2,\tau}^{\leq{d}}\{r,s\}\right)_d$$ converge to $\Phi_{2,\tau}\{r,s\}$ as $d$ tends to $\infty$. Also, $J_{2,\tau}^{\leq d}\{r,s\}$ converges to $J_{2,\tau}\{r,s\}$; thus the residue of $J_{2,\tau}$ is equal to $\Phi_{2,\tau}$
\end{proof}
For an RM-divisor, $\mathcal{D} = \sum_i m_i [\tau_i],$ we will write $$J^\times_\mathcal{D} = \prod _i  \left(J_{\tau_i}^\times\right)^{m_i}.$$
Given a rigid meromorphic theta cohomology class in $H^1(\Gamma, \mathcal{M}^\times/ \mathbb{C}_p^\times)$, we consider lifting the class to a rigid meromorphic cohomology class in $H^1(\Gamma, \mathcal{M}^\times).$ It is not always possible to find such a lift. The RM-divisors for which such a lift exists are analogous to principal divisors on an algebraic curve.

\begin{definition}
Given a rigid meromorphic cohomology class $[\mathcal{J}] \in H^1(\Gamma, \mathcal{M}^\times)$, let $\mathcal{J}$ be a representative meromorphic cocycle $\mathcal{J}$. The value of $\mathcal{J}$ at an RM-point $\tau$ is defined by $$\mathcal{J}[\tau] = \mathcal{J}(\gamma_\tau)(\tau).$$ This value is independent of the choice of representative.
\end{definition}

\subsection{Higher Weight Cocycles}
We now generalize this construction of meromorphic cocycles to higher weight. In doing so, we will replace the weight-$2$ module with $\mathcal{O}(k,\mathcal{H}_K)$.

Fix an even weight $k > 2$.

\begin{definition} \label{def: rm edge function}
For every inert RM-divisor $\mathcal{D}$, define the function $\Phi_{k,\mathcal{D}}\{r,s\} \in C_K^1(P_n)$ as follows: For an edge $e\in \mathcal{T}^1_K\backslash \mathcal{T}^1_{p}$, such that $U_e\cap \mathbb{P}^1(\mathbb{Q}_p) =\emptyset$ let  $$\Phi_{k,\mathcal{D}}\{r,s\}(e) =  \sum_i\sum_{w\in \Sigma_{\tau_i}\{r,s\}\cap U_e}m_i\left[(w,\overline{w})\cdot (r,s)\right]\frac{(T-w)^{n/2}(T-\overline{w})^{n/2}}{(w-\overline{w})^{n/2}},$$ and for $e \in \mathcal{T}^1_{p}$:  $$\Phi_{k,\mathcal{D}}(e) =0 .$$

Otherwise, the function is defined by $\Phi_{k,\mathcal{D}}\{r,s\}(\overline{e}) = -\Phi_{k,\mathcal{D}}\{r,s\}(e)$.
\end{definition}

\begin{proposition}
    For an RM-divisor $\mathcal{D}$, the map $$\{r,s\} \mapsto \Phi_{k,\mathcal{D}}\{r,s\}$$ is a $\Gamma$-invariant modular symbol.
\end{proposition}

\begin{proof}
The proof is exactly the same as for proposition \ref{ModularSymbolproposition}.
\end{proof}
As opposed to the situation in weight $2$, these functions are not necessarily harmonic. From their definition, it follows that $\nabla \Phi_{k,\mathcal{D}}$ is supported on $\mathcal{T}^0_{p}$.

\begin{proposition}
$$\nabla (\Phi_{k,\mathcal{D}}) = \mathrm{Deg}_{k,\mathcal{D}}.$$ 
\end{proposition}
\begin{proof}
For any vertex $v\in \mathcal{T}_K^0\backslash \mathcal{T}_p^0$, there is a unique edge, $e_{v,0}$, which starts at $v$ and points towards the standard vertex and the remaining edges, $e_{v,1},\cdots, e_{v,p^2}$ that start at $v$ satisfy $U_{e_i}\cap \mathbb{P}^1(\mathbb{Q}_p) = \emptyset$, and similarly $U_{\overline{e_{v,0}}} \cap \mathbb{P}^1(\mathbb{Q}_p) = \emptyset$. Since $$U_{\overline{e_{v,0}}} = \bigcup_{i=1}^{p^2}U_{e_i},$$ we see that $$\nabla(\Phi_{k,\mathcal{D}})(v) = 0.$$

Now if $v\in \mathcal{T}^0_p$, then there are exactly $p^2-p$ edges $e_1,\cdots, e_{p^2-p}$ starting at $v$ and satisfying $U_{e_i}\cap \mathbb{P}^1(\mathbb{Q}_p) = \emptyset $. The result follows since an RM-point, $\tau \in \mathcal{H}_p$, reduces to $v$: $$\mathrm{red}_{p}(\tau) =v $$ if and only if $v\in U_{e_i}$ for some $i$.

\end{proof}

\begin{proposition}

    For every inert RM-divisor $\mathcal{D}$ and pair $r,s$, the function $\Phi_{k,\mathcal{D}}\{r,s\}$ is integral; i.e. $$\Phi_{k,\mathcal{D}}\{r,s\}\in C^{1,int}_K(P_n(\mathbb{C}_p)).$$
\end{proposition}

\begin{proof}
    Let $w$ be an inert RM-point, and $e\in \mathcal{T}^1_K$ an edge such that $U_e\cap \mathbb{P}^1(\mathbb{Q}_p) = \emptyset$ and $w\in U_e$. It suffices to show that $$\frac{(T-w)^{n/2}(T-\overline{w})^{n/2}}{(w-\overline{w})^{n/2}} \in P_{n,e}.$$ 
    Let $v_{w}\in \mathcal{T}^0_{p}$ be the reduction of $w$ in $\mathcal{T}_{p}$. Then $e$ lies on the infinite path $(v_w, v_{w,1}, v_{w,2} ,\cdots)$ in $\mathcal{T}_{K}$ which begins at $v_w$ and ends at $w$. We will show that for every vertex $v_{w,a}$, we have $$\frac{(T-w)^{n/2}(T-\overline{w})^{n/2}}{(w-\overline{w})^{n/2}} \in P_{n,v_{w,a}}.$$ 
    First, consider the matrix $\begin{pmatrix} 1&-\overline{w} \\ 1 & -w \end{pmatrix}$; this matrix sends $w\mapsto \infty$, $\overline{w}\mapsto 0$ and $v_w \mapsto v_0$. Moreover, $$T^{n/2}|\begin{pmatrix} 1&-\overline{w} \\ 1 & -w \end{pmatrix} = (-1)^{n/2}\frac{(T-w)^{n/2}(T-\overline{w})^{n/2}}{(w-\overline{w})^{n/2}}.$$ But then, the product $\begin{pmatrix} p^a&0\\0&1  \end{pmatrix}\cdot \begin{pmatrix} 1&-\overline{w} \\ 1 & -w \end{pmatrix} $ sends $v_{w,a}$ to $v_0$ and maps $$T^{n/2}\mapsto (-1)^{n/2}\frac{(T-w)^{n/2}(T-\overline{w})^{n/2}}{(w-\overline{w})^{n/2}}.$$
\end{proof}

\begin{definition}\label{def: deg 0 weight k cocycle}
If $\mathrm{Deg}_{k,\mathcal{D}}$ vanishes, then we say that $\mathcal{D}$ is of strong degree zero. In this case, $\Phi_{k,\mathcal{D}}$ defines a modular symbol valued in bounded harmonic measures and applying the Schneider-Teitelbaum lift, we obtain a modular symbol $$J_{k,\mathcal{D}}\defeq \mathrm{ST}(\Phi_{k,\mathcal{D}})\in\mathrm{MS}^{\Gamma}(\mathcal{O}(k,\mathcal{H}_K)).$$ In fact, $J_{k,\mathcal{D}}$ belongs to $\mathrm{MS}^\Gamma(\mathcal{M}(k; K))$.
\end{definition}

\begin{proposition} \label{weightkexpproposition}
    Let $\mathcal{D} = \sum_i m_i [\tau_i]$ be an RM-divisor of strong degree zero. Then
    $$J_{k,\mathcal{D}}\{r,s\}(z) = \sum_i\sum_{w\in \Sigma_{\tau_i}\{r,s\}}m_i \left[(w,\overline{w})\cdot (r,s)\right] d^{n+1}\left(\frac{(z-w)^{n/2}(z-\overline{w})^{n/2}}{(w-\overline{w})^{n/2}} \log(z-w)\right).$$
\end{proposition}

\begin{proof}
We can truncate $\Phi_{k,\mathcal{D}}$ by $$\Phi^{\leq l}_{k,\mathcal{D}}\{r,s\}(e) =  \sum_i\sum_{w\in \Sigma^{\leq l}_{\tau_i}\{r,s\}\cap U_e}m_i\left[(w,\overline{w})\cdot (r,s)\right]\frac{(T-w)^{n/2}(T-\overline{w})^{n/2}}{(w-\overline{w})^{n/2}}.$$
The boundedness of $\Phi_{k,\mathcal{D}}$ implies that $$\lim_{l \to \infty} \Phi_{k,\mathcal{D}}^{\leq l}\{r,s\}  = \Phi_{k,\mathcal{D}}\{r,s\}.$$ By  Lemma \ref{STFin} we have $$\mathrm{ST}\left(\Phi_{k,\mathcal{D}}^{\leq l}\{r,s\}\right) = \sum_i\sum_{w\in \Sigma_{\tau_i}^{\leq l}\{r,s\}}m_i \left[(w,\overline{w})\cdot (r,s)\right] d^{n+1}\left(\frac{(z-w)^{n/2}(z-\overline{w})^{n/2}}{(w-\overline{w})^{n/2}} \log(z-w)\right),$$
but the Schneider-Teitelbaum lift is a topological isomorphism, and so the limit of the sum above converges and equals $\mathrm{ST}(\Phi_{k,\mathcal{D}})$ as $l$ tends to $\infty$.

\end{proof}

\begin{remark}
If we take the divisor $\mathcal{D} = [\tau] + (-1)^{n/2} [\overline{\tau}]$, then $\mathcal{D}$ is of strong degree zero since for every RM-point, $\mathrm{red}_{p}(w) = \mathrm{red}_{p}(\overline{w})$. It follows from the identity $$d^{n+1}\left(\frac{(z-w)^{n/2}(z-\overline{w})^{n/2}}{(w-\overline{w})^{n/2}}\log\left(\frac{z-w}{z-\overline{w}}\right) \right) = (n!)^2\frac{(w-\overline{w})^{k/2}}{(z-w)^{k/2}(z-\overline{w})^{k/2}}$$ that $J_{k,\mathcal{D}}$ is equal to the meromorphic symbols denoted by $\Phi_\tau$ in Negrini's thesis \cite{NegThesis} and in \cite{Neg1}, Section 1.
\end{remark}

\begin{theorem}\label{classificationtheorem}
\begin{enumerate}
\item Let $J\in\mathrm{MS}^\Gamma(\mathcal{M}(k,\mathcal{H}_p))$ be a weight $k$ rigid meromorphic modular symbol, then for each $r,s$ the divisor of $J\{r,s\}$ is supported on (not necessarily inert) RM-points, and the residues $$\mathrm{res}_w J\{r,s\}(z)(T-z)^n \; dz$$ are multiples of $$\frac{(T-w)^{n/2}(T-\overline{w})^{n/2}}{(w-\overline{w})^{n/2}}.$$ If the divisor is supported on inert RM-points then $\mathrm{Div}(J) = \mathrm{Div}_{k,\mathcal{D}}$ for some RM-divisor $\mathcal{D}$.
\item For any inert RM-divisor $\mathcal{D}$, there exists a rigid meromorphic modular symbol $J\in\mathrm{MS}^\Gamma(\mathcal{M}(k;K))$ such that $$\mathrm{Div}(J) = \mathrm{Div}_{k,\mathcal{D}}.$$ 
\end{enumerate}
\end{theorem}

\begin{proof}
The first part follows from Choie-Zagier \cite{CZ} Lemma 5 as presented in \cite{DV1} Lemma 1.19.

For the second part, we would like to find a harmonic function in $C^1_{K,har}(P_n)$ associated to $\mathcal{D}$. Note that $\Phi_{k,\mathcal{D}}$ is automatically harmonic at every vertex in $v\in \mathcal{T}_K^0\backslash \mathcal{T}_p^0$. We only need to ``fix'' its harmonicity on $\mathcal{T}_p^0$.
It suffices to find some $\chi \in \mathrm{MS}^{\Gamma}(C_p^1(P_n))$ such that $\Phi_{k,\mathcal{D}} + \chi $ is harmonic. That is, $$\nabla\chi\{r,s\} = -\mathrm{Div}_{k,\mathcal{D}}\{r,s\}.$$
By definition, there exists such a $\chi$ if and only if $\mathrm{Div}_{k,\mathcal{D}}$ is in the image of $$MS^\Gamma(C_{p}^1(P_n(\mathbb{C}_p))\xrightarrow{\nabla} MS^\Gamma(C_{p}^0(P_n(\mathbb{C}_p)).$$ But note that $\nabla$ in fact surjective. This can be seen by identifying it with the co-restriction map $$\mathrm{MS}^{\Gamma_0(p)}\left(P_n(\mathbb{C}_p)\right)\longrightarrow  \left(\mathrm{MS}^{\mathrm{SL}_2(\mathbb{Z})}\left(P_n(\mathbb{C}_p)\right)\right)^2,$$
which is known to be surjective.

By the Schneider-Teitelbaum lift we can now define a modular symbol $$J_{\mathcal{D}}:= \mathrm{ST}(\Phi_{k,\mathcal{D}} + \chi)\in \mathrm{MS}^\Gamma(\mathcal{O}(k,\mathcal{H}_K)),$$ and by Lemma \ref{merextLem} the functions $J_{\mathcal{D}}\{r,s\}$ extend uniquely to meromorphic functions on $\mathcal{H}_p$, and therefore $J_{\mathcal{D}}$ extends to a modular symbol in $\mathrm{MS}^\Gamma(\mathcal{M}(k,\mathcal{H}_p))$.
\end{proof}

\begin{remark}
The same idea can also be used to treat ramified RM-points as well, although the combinatorics of the Bruhat--Tits tree are slightly different. This treatment will be presented in the author's forthcoming thesis.
\end{remark}

\section{Obstruction to lifting and the Green's function}
In the multiplicative setting there is an exact sequence $$H^1(\Gamma, \mathcal{M}^\times)\longrightarrow H^1(\Gamma, \mathcal{M}^\times/\mathbb{C}_p^\times) \longrightarrow H^2(\Gamma, \mathbb{C}_p^\times), $$ where the map to $H^2(\Gamma, \mathbb{C}_p^\times)$ is the obstruction to lifting a theta cocycle class to a meromorphic cocycle class. In higher weight, we work instead with the additive counterpart to this obstruction map. That is, we further compose with a homomorphism, $$\mathbb{C}_p^\times \rightarrow \mathbb{C}_p$$ to obtain a compatible collection of additive obstruction maps. In what follows, we will define the morphisms $$H^1(\Gamma, \mathcal{M}(k, \mathcal{H}_p; K)) \xrightarrow{\mathrm{Obs}_{\mathscr{L}}} H^2(\Gamma, P_n(\mathbb{C}_p))$$ for each $\mathscr{L}\in \mathbb{C}_p$, 
and a corresponding one to the $p$-adic valuation 
$$H^1(\Gamma, \mathcal{M}(k, \mathcal{H}_p; K)) \xrightarrow{\mathrm{Obs}_o} H^2(\Gamma, P_n(\mathbb{C}_p)). $$

\begin{definition}[$\log_{\mathscr{L}}$ obstruction] \label{def: log obstruction}
For every $\mathscr{L}\in\mathbb{C}_p$ the $\Gamma$-cohomology of the short exact sequence \begin{equation}\label{lesL}0 \longrightarrow  P_n(\mathbb{C}_p)\longrightarrow  \mathcal{O}_\mathscr{L}(2-k,\mathcal{H}_F) \longrightarrow  \mathcal{O}(k,\mathcal{H}_F)\longrightarrow  0\end{equation} gives rise to a $\delta$-map $$H^1(\Gamma,\mathcal{O}(k,\mathcal{H}_F)) \longrightarrow  H^2(\Gamma, P_n(\mathbb{C}_p))$$ which we denote by $\mathrm{Obs}_{\mathscr{L}}$.
\end{definition}

\begin{definition}[$\mathrm{ord}_p$ obstruction]
 Define $C_{F,har}^0(P_n(\mathbb{C}_p))$ to be the set of functions $f: \mathcal{T}^0_F \longrightarrow  P_n(\mathbb{C}_p)$ such that for every vertex, $v$, we have $$ \sum_{\substack{v'\in \mathcal{T}_F^0\\(v,v')\in \mathcal{T}_F^1}}\left(f(v)-f(v') \right) = 0.$$
 There is an analogous short exact sequence to (\ref{lesL}) for the $p$-adic valuation. \begin{equation}\label{ordSES}0\longrightarrow  P_n(\mathbb{C}_p)\longrightarrow  C^0_{har}(P_n(\mathbb{C}_p))\longrightarrow  C^1_{har}(P_n(\mathbb{C}_p))\longrightarrow  0,\end{equation}
where the map  $C^0_{har}(P_n(\mathbb{C}_p))\longrightarrow  C^1_{har}(P_n(\mathbb{C}_p))$ is given by $$f \mapsto \left((v,v')\mapsto f(v)-f(v')\right).$$
The long exact sequence in group cohomology from (\ref{ordSES}) gives a map $$\delta:H^1(\Gamma,C^1_{har}(P_n(\mathbb{C}_p)))\longrightarrow  H^2(\Gamma, P_n(\mathbb{C}_p)).$$ Precomposing with the residue map provides us with the obstruction $$\mathrm{Obs}_{o}: H^1(\Gamma, \mathcal{O}(k,\mathcal{H}_F))\longrightarrow  H^2(\Gamma, P_n).$$

\end{definition}

The map $\mathrm{Obs}_o$ can be described more explicitly as follows: Given a cocycle $$\varphi \in H^1(\Gamma, C^1_{har}(P_n(\mathbb{C}_p))),$$ a particular choice of a lift to $C_{har}^0(P_n)$ is given by the map $$\gamma \in \Gamma \mapsto \left( v\mapsto \sum_{e\in\text{path}(v_0,v)}\phi(\gamma)(e)\right),$$ which is not necessarily a cocycle. The $\delta$-map to $H^2(\Gamma,P_n)$ with this choice of lift yields the $2$-cocycle $$\mathrm{Obs}_{o}(\varphi)(\gamma_1, \gamma_2) = \sum_{e\in \text{path}(v_0,\gamma_2^{-1}v_0)}\varphi(\gamma_1)(e)|\gamma_2.$$
\begin{proposition}
\label{deShalitIDproposition}

For any $\mathscr{L}_1,\mathscr{L}_2\in\mathbb{C}_p$ the various obstruction maps are related by
$$\mathrm{Obs}_{\mathscr{L}_1} - \mathrm{Obs}_{\mathscr{L}_2} = (\mathscr{L}_1 - \mathscr{L}_2) \mathrm{Obs}_{o}.$$
\end{proposition}

We will apply the following lemma of de Shalit: 

\begin{lemma}[\cite{dSBounded} Lemma 4.3 ]
\label{deShalitlemma}
For any $f\in \mathcal{O}(k,\mathcal{H}_F)$, let $e = (v_1,v_2)\in \mathcal{T}_F^1$ be an edge. Fix two lifts $F^{\mathscr{L}_1}\in \mathcal{O}_{\mathscr{L}_1}\otimes P_n$ and $F^{\mathscr{L}_2}\in \mathcal{O}_{\mathscr{L}_2}\otimes P_n$ of $f(z)(T-z)^{n}dz$, and let $\varphi_f \in C_{har}$ be the harmonic function associated to $f$. Then for any $z_i \in \mathrm{red}^{-1}(v_i)$ $$(F^{\mathcal{L}_1}(z_1) - F^{\mathcal{L}_1}(z_2)) - (F^{\mathcal{L}_2}(z_1) - F^{\mathcal{L}_2}(z_2))  = (\mathscr{L}_1 - \mathscr{L}_2)\varphi_f(e).$$
\end{lemma}

\begin{proof}{Of proposition \ref{deShalitIDproposition}}

Following de Shalit's proof (\cite{dSBounded} Proposition 4.2), fix $\tau\in \mathcal{H}_p$ with $\mathrm{red}_K(\tau) = v_0$ and fix two lifts $$J^{\mathscr{L}_i}:\Gamma \longrightarrow  \mathcal{O}_{\mathcal{L}_i}\otimes P_n(\mathbb{C}_p),\; i\in \{1,2\},$$ of  the cocycle $\gamma \mapsto J(\gamma)(z)(T-z)^n\;dz.$ The value of the $2$-cocycle $$\mathrm{Obs}_\mathscr{L_1}(J)(\gamma_1,\gamma_2) -\mathrm{Obs}_\mathscr{L_2}(J)(\gamma_1,\gamma_2)$$ is given by  \begin{align*}&\left(J^{\mathscr{L}_1}(\gamma_1\gamma_2)(\tau)-\left(J^{\mathscr{L}_1}(\gamma_1)\mid \gamma_2\right)(\tau) - J^{\mathscr{L}_1}(\gamma_2)(\tau)\right)\\ 
-&\left(  J^{\mathscr{L}_2}(\gamma_1\gamma_2)(\tau)-\left(J^{\mathscr{L}_2}(\gamma_1)\mid \gamma_2\right)(\tau) - J^{\mathscr{L}_2}(\gamma_2)(\tau)\right).
\end{align*}
The function $F:\Gamma \rightarrow  P_n(\mathbb{C}_p)$ defined by $$F(\gamma) \defeq J^{\mathscr{L}_1}(\gamma)(\tau) - J^{\mathscr{L}_2}(\gamma)(\tau) $$ produces a co-boundary $$ \gamma_1,\gamma_2 \mapsto F(\gamma_1 \gamma_2)(\tau) - F(\gamma_1)|\gamma_2 - F(\gamma_2).$$
By subtracting this co-boundary, we obtain a co-homologous cocycle to $\mathrm{Obs}_{\mathscr{L}_1} - \mathrm{Obs}_{\mathscr{L}_2}$ given by $$\gamma_1, \gamma_2\mapsto \left[\left(J^{\mathscr{L}_1}(\gamma_1)(\gamma_2\tau)  - J^{\mathscr{L}_1}(\gamma_1)(\tau)\right) - \left(J^{\mathscr{L}_2}(\gamma_1)(\gamma_2\tau)  - J^{\mathscr{L}_2}(\gamma_1)(\tau)\right) \right] \mid \gamma_2.$$

Let $(w_1, \cdots, w_h)$ be the path on $\mathcal{T}_K$ connecting $w_1$ and $\gamma w_h$ and pick a collection of points $(z_1,\cdots, z_h)$ such that $z_1 = \tau$, $z_h = \gamma \tau$ and for all $i$, $\mathrm{red}_K(z_i) = w_i$, then applying Lemma \ref{deShalitlemma} to the telescoping sum $$\sum_{i=1}^{h-1}\left[\left(J^{\mathscr{L}_1}(\gamma_1)(z_{i+1})  - J^{\mathscr{L}_1}(\gamma_1)(z_i)\right) - \left(J^{\mathscr{L}_2}(\gamma_2)(z_{i+1})  - J^{\mathscr{L}_2}(\gamma_1)(z_i)\right) \right] \mid \gamma_2,$$
we obtain $$ \left(\mathrm{Obs}_{\mathscr{L}_1} - \mathrm{Obs}_{\mathscr{L}_2}\right)(J)(\gamma_1,\gamma_2) + \text{ (co-boundary) }=\sum_{i=1}^{h-1} \varphi(\gamma_1)((v_i,v_{i+1}))|\gamma_2 = \mathrm{Obs}_{o}(J)(\gamma_1,\gamma_2).$$

\end{proof}

\begin{lemma}
 Let $\mathcal{D}$ be an $RM$-divisor. If $\mathrm{Deg}_{k,\mathcal{D}}$ is trivial, then so is $\mathrm{Obs}_{o}(J_{k,\mathcal{D}}).$
\end{lemma}

\begin{proof}
    Let $\mathcal{D}$ be an $RM$-divisor, then $J_{k,\mathcal{D}}$ is chosen so that for every edge $e\in \mathcal{T}_p^1$ the residues vanish, that is for every $\gamma \in \Gamma$ we have $$\mathrm{res}\left(J_{k,\mathcal{D}}(\gamma)\right)(e) = 0.$$
    But, notice that the path $(v_0,\gamma v_0)$ lies completely in $\mathcal{T}_p^1$, so from the definition of $\mathrm{Obs}_{o}$ as a sum over the path $(v_0, \gamma v_0)$ it is immediate that $\mathrm{Obs}_{o}(J_{k,\mathcal{D}})$ is trivial.
\end{proof}

\begin{lemma}\label{lem: weight k cocycle}
    For every RM-divisor $\mathcal{D}$, there exists a unique cohomology class in $ H^1(\Gamma, \mathcal{M}(k;K))$, which we will denote $J_{k,\mathcal{D}}$, such that $\mathrm{div}(J_{k,\mathcal{D}}) = \mathrm{Div}_{k,\mathcal{D}}$ and  $\mathrm{Obs}_{o}(J_{k,\mathcal{D}}) = 0 $.
\end{lemma}

\begin{proof}
By Theorem \ref{classificationtheorem}, pick an arbitrary cohomology class $J'_{k,\mathcal{D}}$ which has divisor $\mathcal{D}$. The map $\mathrm{Obs}_{o}:H^1(\Gamma, \mathcal{O}(k, \mathcal{H}_p)) \longrightarrow  H^2(\Gamma,P_n)$ is an isomorphism by Proposition \ref{obstructionIso}. Therefore there exists a unique cohomology class $J \in H^1(\Gamma, \mathcal{O}(k, \mathcal{H}_{p}))$ such that $\mathrm{Obs}_{o}(J) = \mathrm{Obs}_{o}(J'_{k,\mathcal{D}})$. Therefore, $$\mathrm{Obs}_{o}(J'_{k,\mathcal{D}} - J) = 0,$$ and the divisor of $J'_{k,\mathcal{D}} - J$ is $\mathrm{Div}_{k, \mathcal{D}}$.
\end{proof}

\begin{lemma} \label{wt-ncoLem}
If $\mathcal{D}$ is of degree $0$, then for each $r,s$ define a function $J_{k,\mathcal{D}}^\mathscr{L}$ by $$J_{k,\mathcal{D}}^{\mathscr{L}}\{r,s\}(z) = \left\langle J_{k,\mathcal{D}}, l_z(t)\right\rangle_M,$$ with $$l_z(t) = \begin{cases}-\frac{1}{n!}(t-z)^n \log_{\mathscr{L}}(t-z) & \text{if } |t|\leq1 \\ -\frac{1}{n!}(t-z)^n \log_{\mathscr{L}}\left(1-\frac{z}{t}\right) & \text{if }  |t|>1 \end{cases} \; \; \in \Sigma(k,K).$$ Then $J_{k,\mathcal{D}}^\mathscr{L}\{r,s\}$ belongs to $\mathcal{M}_\mathscr{L}(2-k; K)$ and satisfies $d^{n+1}\mathcal{J}^\mathscr{L}_{k,\mathcal{D}}\{r,s\} = J_{k,\mathcal{D}}\{r,s\}$.
\end{lemma}

\begin{proof}
This follows from writing out the Mittag-Leffler expansion of $J_{k,\mathcal{D}}^{\mathscr{L}}\{r,s\}(z)$.
For a point $t\in \mathbb{P}^1(K)$ close enough to $a\in \mathcal{R}_l$, a point of the set of representatives of level $l$, we can write $$(t-z)^n\log_{\mathscr{L}}(t-z) = (t-z)^n\log_\mathscr{L}(z-a) + (t-z)^n \log_{\mathscr{L}}\left(1-\frac{t-a}{z-a}\right) $$

$$ = (t-z)^n\log_{\mathscr{L}}(z-a) -  (t-z)^n\sum_{i=1}^\infty \frac{(t-a)^i}{i(z-a)^i}.$$  
The Morita pairing is then computed by integrating the various factors $(t-a)^i$ as described in \cite{MTT}, \cite{Orton} or \cite{DT}. Differentiating this expansion with respect to $z$, we obtain exactly the expression for $$-n!\left\langle J_{k,\mathcal{D}}, \frac{1}{t-z}\right\rangle_M= -n!J_{k,\mathcal{D}}(z).$$
\end{proof}

\begin{remark}
The choice of function $l_z(t)$ is informed by the identity $$\frac{d^{n+1}}{dz^{n+1}}\left((t-z)^n\log(t-z)\right) = \frac{(-1)^{n-1}n!}{(t-z)}.$$
\end{remark}

\begin{remark}
    From the its expression as an infinite sum, one is tempted to find a primitve of $J_{k,\mathcal{D}}$ term by term and add them in a sum of the form $$\sum_{w\in \Sigma_{\tau}\{r,s\}}[(w,\overline{w})\cdot (r,s)]\frac{(z-w)^{n/2}(z-\overline{w})^{n/2}}{(w-\overline{w})^{n/2}}\log_{\mathscr{L}}(z-w).$$
    However, this sum does not converge. It only converges ``up to polynomials". That is, on every affinoid, the coefficients of a Mittag-Leffler expansion of this sum converge except for the terms of the form $c_iz^i$ where $0\leq i \leq n$. 

\end{remark}
\begin{lemma} \label{expl-nLem}
If $\mathcal{D}$ is of degree $0$, then for each $r,s$ the infinite sum $$ \frac{1}{n!}\sum_i \sum_{w\in \Sigma_{\tau_i}\{r,s\}}m_i[(w,\overline{w}).(r,s)] \left( \frac{(z-w)^{n/2}(z-\overline{w})^{n/2}}{(w-\overline{w})^{n/2}}\log_{\mathscr{L}}(t_w(z))+p_w(z)\right)$$   converges to $J_{k,\mathcal{D}}^{\mathscr{L}}\{r,s\}(z)$. Where $t_w(z)$ is as in \emph{(\ref{twFun})} and if $|w|\leq 1$ then $$p_w(z) = -\sum_{r=1}^{n/2}\left(\sum_{i=1}^{r-1}\left(\frac{(-1)^{n-i}\binom{n}{i}}{(r-i)}\right)\frac{(-1)^{r}\binom{n-r}{n/2-r}(w-\overline{w})^{r-n/2}}{\binom{n}{n/2}}(z-w)^{n-r}\right),$$
and otherwise 
$$p_w(z) = -\sum_{r=1}^{n/2}\left(\sum_{i=1}^{r-1}\left(\frac{(-1)^{n-i}\binom{n}{i}}{(r-i)}\right)\frac{(-1)^{r}\binom{n-r}{n/2-r}(w-\overline{w})^{r-n/2}}{\binom{n}{n/2}}(z-w)^{n-r}\right)$$
$$ +\sum_{i=n/2+1}^n(-1)^i\binom{n}{i} \left(\sum_{j=1}^{i-n/2} \frac{(-1)^j(-1)^{n-i+j}\binom{i-j}{i-j-n/2}(w-\overline{w})^{n/2-i+j}}{\binom{n}{n/2}jw^j} \right)(z-w)^i.$$
\end{lemma}
\begin{proof}

    The proof is identical to that of proposition \ref{weightkexpproposition} except we have to expand the functions of the form $$(t-z)^n\log_\mathscr{L}(t-z) = (t-z)^n\log_{\mathscr{L}}(z-w) + (t-z)^n\log_{\mathscr{L}}\left(1-\frac{t-w}{z-w}\right).$$ and integrate them in a neighborhood of $\tau$.
\end{proof}

\subsection{Duals and the Abel-Jacobi Map}
The short exact sequence $$ 0 \longrightarrow \Sigma(k,F)\longrightarrow \Sigma_{\mathscr{L}}(k,F)\longrightarrow P_n(\mathbb{C}_p)\longrightarrow 0$$ leads to a long exact sequence in $\Gamma$-homology \begin{equation}\label{homology} \cdots \leftarrow H_1(\Gamma, \Sigma_{\mathscr{L}}(k,F)) \leftarrow H_1(\Gamma, \Sigma(k,F)) \leftarrow H_2(\Gamma,P_n(\mathbb{C}_p))\leftarrow \cdots.  \end{equation}

For every $\mathscr{L}$, let $i_{\mathscr{L}}$ be the inclusion \begin{align*}i_{\mathscr{L}}:\; &\mathrm{Div}(\mathcal{H}_p)\otimes P_n \rightarrow \Sigma_{\mathscr{L}}(k,\mathbb{Q}_p)\\  &(x)\otimes P \mapsto P(t)\log_{\mathscr{L}}\left(t-x \right).\end{align*}
The map $i_{\mathscr{L}}$ maps $\mathrm{Div}_0(\mathcal{H}_p)\otimes P_n$ to $\Sigma(k, \mathbb{Q}_p)$. 
The combination of $i_\mathscr{L}$ with the Morita pairing defines a pairing $$\langle\; ,\; \rangle_{M}: H^1(\Gamma,\mathcal{O}(k,\mathcal{H}_p))\times H_1(\Gamma, \mathrm{Div}_0(\mathcal{H}_p))\longrightarrow \mathbb{C}_p. $$ 
The Breuil pairing extends this pairing to $$\langle \; , \; \rangle_{\mathscr{L}}: H^1(\Gamma,\mathcal{M}_{\mathscr{L}}(2-k; K))\times H_1(\Gamma, \mathrm{Div}(\mathcal{H}_p))\longrightarrow \mathbb{C}_p.$$
Further, there is also a pairing for the $p$-adic valuation $$\langle \; , \; \rangle_{o}:H^1(\Gamma,\mathcal{O}(k,\mathcal{H}_p))\times H_1(\Gamma, \mathrm{Div}_0(\mathcal{H}_p^\dagger)\otimes P_n(\mathbb{C}_p))\longrightarrow  \mathbb{C}_p,$$ 
which is induced from a pairing $$\mathcal{O}(k,\mathcal{H}_p)\times \mathrm{Div}_0(\mathcal{H}_p^\dagger)\otimes P_n \longrightarrow \mathbb{C}_p,$$ that is defined as follows:
Given $f \in \mathcal{O}(k,\mathcal{H}_p)$, its associated harmonic function $\phi_f\in C_{\text{har}}(P_n)$, and for $$D = \sum_i \tau_i\otimes p_i \in \mathrm{Div}_0(\mathcal{H}_p^\dagger)\otimes P_n,$$ set $$\mathrm{red}(D) \defeq \sum_i w_i\otimes p_i$$ to be the reduction of $D$ in $\mathrm{Div}_0(\mathcal{T}_p^0)\otimes P_n$. Then the pairing between $f$ and $D$ is equal to the sum $$\langle f, D \rangle _o \defeq \sum_i\sum_{e\in \text{path}(v_0,w_i)}\langle \phi_f(e),p_i\rangle_{P_n}.$$
This pairing can be extended to $$\langle \cdot , \cdot \rangle _o :C^0_{har}(P_n(\mathbb{C}_p))\times  \mathrm{Div}(\mathcal{H}_p) \longrightarrow  \mathbb{C}_p$$ as follows: If $\phi \in C^0_{har}(P_n(\mathbb{C}_p))$ and $D\in \mathrm{Div}(\mathcal{H}_p)\otimes P_n$ again write $\mathrm{red}(D) = \sum_i w_i\otimes p_i$ then set $$\langle \phi,D \rangle_o \defeq \sum_i \langle\phi(w_i), p_i\rangle_{P_n}.$$

\begin{definition}

    For $\sigma$ an RM-point whose fundamental stabilizer is $\gamma_\sigma\in \Gamma$. The Stark-Heegner cycle associated to $\sigma$ is $$y_{k,\sigma} = \gamma_\sigma\otimes (\sigma)\otimes \frac{(T-\sigma)^{n/2}(T-\overline{\sigma})^{n/2}}{(\sigma - \overline{\sigma})^{n/2}} \in H_1(\Gamma, \mathrm{Div}(\mathcal{H}_p)\otimes P_n(\mathbb{C}_p)).$$
    Given a cocycle $\mathcal{J}\in H^1(\Gamma,\mathcal{O}_\mathscr{L}(k, \mathcal{H}_{p}))$ we will denote the pairing $\langle \mathcal{J}, y_{k,\sigma}\rangle_B$ by $$\mathcal{J}[\sigma].$$ This pairing is given by $$\left\langle \mathcal{J}(\gamma_\sigma)(z), \frac{(t-\sigma)^{n/2}(t-\overline{\sigma})^{n/2}}{(\sigma - \overline{\sigma})^{n/2}}\log_{\mathscr{L}}(t-\sigma)\right\rangle _B.$$

\end{definition}

\begin{proposition}
\label{commDiag}
The pairings described above are related by the following commutative diagrams
\begin{center}
\begin{tikzcd}
H^1(\Gamma,\mathcal{O}_\mathscr{L}(2-k))^\vee  & H^1(\Gamma, \mathcal{O}(k))^\vee \ar[l,"(d^{n+1})^\vee"] & H^2(\Gamma, P_n(\mathbb{C}_p))^\vee \ar[l,"\mathrm{Obs}_\mathscr{L}^\vee"] \\ 

H_1(\Gamma, \mathrm{Div}(\mathcal{H}_p)\otimes P_n(\mathbb{C}_p)) \ar[u] & H_1(\Gamma, \mathrm{Div}_0(\mathcal{H}_p)\otimes P_n(\mathbb{C}_p)) \ar[l] \ar[u] & H_2(\Gamma,P_n(\mathbb{C}_p))\ar[l] \ar[u]
\end{tikzcd}

\end{center}

\vspace{0.2in}
and

\begin{center}
\vspace{0.2in}

\begin{tikzcd}
    H^1(\Gamma,C_{har}^0(\mathcal{P}_n(\mathbb{C}_p)))^\vee & H^1(\Gamma, C_{har}^1(P_n(\mathbb{C}_p)))^\vee\ar[l]  & H^2(\Gamma, P_n(\mathbb{C}_p))^\vee \ar[l, "\mathrm{Obs}_{o}^\vee"] \\ 
    H_1(\Gamma, \mathrm{Div}(\mathcal{H}_p)\otimes P_n(\mathbb{C}_p)) \ar[u]& H_1(\Gamma, \mathrm{Div}_0(\mathcal{H}_p)\otimes P_n(\mathbb{C}_p)) \ar[l] \ar[u]& H_2(\Gamma, P_n(\mathbb{C}_p)) \ar[l] \ar[u]
    \end{tikzcd}
\end{center}
\end{proposition}

\begin{proposition}
The map $H_2(\Gamma, P_n(\mathbb{C}_p))\longrightarrow  H^2(\Gamma, P_n(\mathbb{C}_p))^\vee$ is an isomorphism.
\end{proposition}

\begin{proof}
This follows from the Universal Coefficient Theorem.
\end{proof}

\begin{lemma}
    The map $$H_1(\Gamma, \mathrm{Div}_0(\mathcal{H}_p)\otimes P_n(\mathbb{C}_p)) \longrightarrow  H_1(\Gamma, \mathrm{Div}(\mathcal{H}_p)\otimes P_n(\mathbb{C}_p))$$ is surjective.
    \end{lemma}

\begin{proof}
The result is a consequence of the vanishing of $H_1(\Gamma, P_n(\mathbb{C}_p))$, which follows since $H^1(\Gamma, P_n)$ vanishes by \cite{RS} Lemma 3.10 and the Universal Coefficient Theorem provides a perfect pairing between $H_1(\Gamma,P_n)$ and $H^1(\Gamma,P_n)$.
\end{proof}

\begin{corollary}\label{Unitarycorollary}
Let $\mathcal{U} = \ker \left( H_1(\Gamma,\mathrm{Div}_0(\mathcal{H}_p)\otimes P_n(\mathbb{C}_p)) \longrightarrow  H^1(\Gamma, C_{har}^1(P_n(\mathbb{C}_p)))\right)$, then
$$H_1(\Gamma,\mathrm{Div}_0(\mathcal{H}_p)\otimes P_n(\mathbb{C}_p)) = \delta (H_2(\Gamma, P_n(\mathbb{C}_p))) \oplus \mathcal{U}.$$

\end{corollary}

\begin{definition} 
It follows from Corollary \ref{Unitarycorollary} that for each $D\in H_1(\Gamma, \mathrm{Div}(\mathcal{H}_p)\otimes P_n(\mathbb{C}_p))$ there exists a unique unitary divisor $D^\sharp \in \mathcal{U}$ which maps to $D$.
\end{definition}

We would like to slightly extend the pairings described above to pair a meromorphic cocycle $\mathcal{J} \in H^1(\Gamma, \mathcal{M}_{\mathscr{L}}(2-k,\mathcal{H}_p;K))$ with a divisor that is supported away from the poles of $\mathcal{J}$. This is analagous to evaluating multiplicative rigid cocycles at RM-points. 

\begin{definition} 
Given $f\in \mathcal{O}(k,\mathcal{H}_K)^b$ and $\tau\in \mathbb{P}^1(K)\backslash\{\infty\}$. We say that $f$ has no residues near $\tau$ if there exists an open neighborhood $U\subset \mathbb{P}^1(K)$ of $\tau$ such that for every edge $e$ with $U_e\subset U$, the residue vanishes: $$\mathrm{res}(f)(e) = 0.$$ 
We also say in this case that $f$ has no residues in $U$. For $F\in \mathcal{O}_\mathscr{L}(2-k,\mathcal{H}_F)$ we say $F$ has no residues near $\tau$ if $d^{n+1}F$ has no residues near $\tau$. 
Finally, we extend these definitions to finite divisors, to say that $f$ has no residues on a divisor $D$ on $\mathrm{Div}(\mathbb{P}^1(F))$ if $f$ has no residues near each point in the support of $D$. 
\end{definition}

\begin{proposition}
If $f\in \mathcal{O}_\mathscr{L}(2-k,\mathcal{H}_K)$, $\tau \in \mathbb{P}^1(K)\backslash \{\infty \}$ and $P\in P_n(\mathbb{C}_p)$ such that $f$ has no residue at $\tau$, then the pairing $$\langle f,\alpha_{\tau\otimes P, U}(t)\rangle_B $$ does not depend on the choice of $U$. Where $U$ is any neighborhood of $\tau$ on which $f$ has no residues and $\alpha_{\tau\otimes P, U}$ is defined by $$\alpha_{\tau\otimes P,U}(t) = \begin{cases}0 & \text{if } t\in U, \\ P(t)\log_\mathscr{L}(t-\tau) & \text{otherwise}. \end{cases}$$
\end{proposition}

\begin{proof}
Given any two such neighborhoods, $U_1$, and $U_2$,  the difference $$\alpha_{\tau\otimes p,U_1}-\alpha_{\tau \otimes P,U_2}$$ belongs to  $\Sigma(k,K)$ and is supported on $U_1\cup U_2$, then by (2) of Proposition \ref{breuilPropertiesProp} $$\langle f,\alpha_{\tau\otimes P,U_1} - \alpha_{\tau\otimes P, U_2}\rangle_B = \langle d^{k-1}f, \alpha_{\tau\otimes P,U_1} - \alpha_{\tau\otimes P, U_2} \rangle_M = 0.$$ 
\end{proof}

We will still use the notation $\langle f, P(t)\log_{\mathscr{L}}(t-\tau)\rangle_B$ to refer to $\langle f, \alpha_{\tau\otimes P,U}\rangle_B $ for some choice of open set $U$. 

The $p$-adic valuation pairing can also be similarly extended: If $f\in \mathcal{O}(k,\mathcal{H}_F)$ has no residues near a divisor, $$D = \sum_i (\tau_i)\otimes p_i \in \mathrm{Div}_0(\mathbb{P}^1(K))\otimes P_n(\mathbb{C}_p).$$ Then for each $i$ the sum $$\sum_{e\in \text{path}(v_0, \tau_i)} \langle \phi_f(e),p_i\rangle$$ is finite since $\phi_f(e)$ is eventually $0$ as the edges approach $\tau_i$.  So we set $$\langle f, D\rangle_{o} \defeq \sum_i \sum_{e\in\text{path}(v_0,\tau_i)} \langle \phi_f(e),p_i\rangle.$$

\begin{definition}
For a function $\phi \in C^0_{har}(P_n)$ we will say that it has no residues near $\tau$ if there exists an open neighborhood $U$ of $\tau$, such that for every $w\in U$, and for every path $(w_0, w_1, \cdots)$ which tends to $w$ the value, $\phi(w_i)$, is eventually constant, and we will denote this constant value by $\phi(w)$. In which case, we define $$\langle \phi ,(\tau)\otimes P\rangle_o \defeq \langle \phi(\tau), P \rangle_{P_n}.$$
\end{definition}

\begin{lemma}
\label{coboundaryLem}
Fix $$\sum_i \gamma_i \otimes \sum_j \left(\sigma_{ij}\otimes P_{ij}\right) \in Z_1(\Gamma, \mathrm{Div}(\mathcal{H}_p^\dagger)\otimes P_n)$$ and $F\in \mathcal{O}_{\mathscr{L}}(2-k,\mathcal{H}_K)$ such that $\sigma_{ij}\in \mathbb{P}^1(K)$, and for each $i$, the translated function $F\vert \gamma_i$ has no residues near $\sigma_{ij}$ for all $j$. Then $$\sum_i \left\langle  f\vert \gamma_i , \sum_j \sigma_{ij}\otimes P_{ij}\right\rangle _B= \sum_i \left\langle f, \sum_j \sigma_{ij}P_{ij}\right\rangle _B .$$

\end{lemma}

\begin{proof}
Unpacking the definitions we have 
\begin{align*}
    &\sum_i \left\langle  f\vert \gamma_i , \sum_j \sigma_{ij}\otimes P_{ij}\right\rangle
    = \sum_i \left\langle f\vert \gamma_i , \sum_j \alpha_{\sigma_{ij}, P_{ij}, U_{ij}}\right\rangle_B \\ 
    =& \sum_i \left\langle f,\sum_j \alpha_{\sigma_{ij}, P_{ij}, U_{ij}}\vert \gamma_i^{-1}\right\rangle_B
    = \sum_i \left\langle f,\sum_j \alpha_{\gamma_i \sigma_{ij}, P_{ij}\vert \gamma_i^{-1}, \gamma_i U_{ij}} \right\rangle_B\\ 
    =& \sum_i \left\langle f,\sum_j \alpha_{\sigma_{ij},P_{ij},U_{ij}'}\right\rangle_B.
\end{align*}
The final equality follows from the condition on the divisor belonging to $ Z_1(\Gamma, \mathrm{Div}(\mathcal{H}_p^\dagger)\otimes P_n)$.
\end{proof}

\begin{definition}
    It follows from Lemma \ref{coboundaryLem} that the extension of the Breuil pairing to meromorphic functions does not depend on the choice of cocycle class since it vanishes on coboundaries. The value $$\left\langle\mathcal{J}(\gamma_\tau), \frac{(t-\sigma)^{n/2}(t-\overline{\sigma})^{n/2}}{(\sigma-\overline{\sigma})^{n/2}}\log_\mathscr{L}(t-\sigma)\right\rangle _B$$ depends only on the cocycle class in $\mathcal{J} \in H^1(\Gamma,\mathcal{O}_{\mathscr{L}}(2-k,\mathcal{H}_K))$. We will denote this value by $\mathcal{J}[\tau]$

    Similarly, for $J\in H^1(\Gamma, \mathcal{O}(k, \mathcal{H}_K))$ and a degree $0$ divisor $D \in H_1(\Gamma, \mathrm{Div}_0(\mathcal{H}_p)\otimes P_n)$. We will denote the associated pairing by $J[D]$.
\end{definition}

\subsection{Higher Green's Functions}\label{sec: higher Green's Functions}

Recall that $\mathcal{RM}^{\text{in}}$ denotes the free $\mathbb{Z}$-module of orbits of $RM$-points on $\mathcal{H}_p$ which are inert. For any $\mathcal{D} \in \mathcal{RM}^{\text{in}}$ we have obtained a uniquely defined cocycle $$J_{k,\mathcal{D}}\in H^1(\Gamma,\mathcal{O}(k,\mathcal{H}_K)),$$ whose divisor is $\mathrm{Div}_{k,\mathcal{D}}$ and such that $\mathrm{Obs}_o(J_{k,D})$ is trivial. We also associate to each such divisor $\mathcal{D}$ a homology class $$y_{k,\mathcal{D}} \in H_1(\Gamma, \mathrm{Div}(\mathcal{H}_p^\dagger)\otimes P_n),$$ and by Lemma \ref{Unitarycorollary} a unique unitary lift $$y_{k,D}^\sharp \in \mathcal{U} \subset H_1(\Gamma, \mathrm{Div}_0(\mathcal{H}_p^\dagger)\otimes P_n).$$
\begin{definition}
For a pair of inert RM-divisors $\mathcal{D_1},\mathcal{D}_2 \in \mathcal{RM}^{\text{in}}$ define the higher Green's function by $$G_k(\mathcal{D}_1,\mathcal{D}_2) = J_{k,\mathcal{D}_1}[y_{k,\mathcal{D}_2}^\sharp].$$
\end{definition} 

\begin{proposition}\label{prop: principal values}
If there exists a lift $\mathcal{J}_{k,\mathcal{D}_1}^\mathscr{L}\in H^1(\Gamma, \mathcal{O}_{\mathscr{L}}(2-k,\mathcal{H}_K))$ such that $$d^{k-1}\mathcal{J}^{\mathscr{L}}_{k,\mathcal{D}_1} = J_{k,\mathcal{D}_1},$$ then $$G_{k}(\mathcal{D}_1,\mathcal{D}_2) = J_{k,\mathcal{D}_1}[y_{k,\mathcal{D}_2}^\sharp]=  \mathcal{J}^{\mathscr{L}}_{k,\mathcal{D}_1}\left[y_{k,\mathcal{D}_2}\right].$$
\end{proposition}

\begin{proof}
    This follows from the commutativity of the diagram in Proposition \ref{commDiag}.
\end{proof}

\section{Conjectures and Computations}

\subsection{Efficient computations of principal values}

In this section, we will describe the computations of values $$G_{k}(\mathcal{D}_1,\mathcal{D}_2),$$ where $\text{Deg}_{k, \mathcal{D}_1}$ is trivial and $\mathcal{D}_1$ is principal; that is, for every $\mathscr{L}\in \mathbb{C}_p$, there exists $\mathcal{J}_{k,\mathcal{D}_1}^{\mathscr{L}} \in H^1(\Gamma,\mathcal{M}_{\mathscr{L}}(2-k ; K))$ such that $$d^{k-1}\mathcal{J}_{k,\mathcal{D}_1}^{\mathscr{L}} = J_{k,\mathcal{D}_1}.$$ 

We follow a similar approach to that of Darmon-Vonk \cite{DV1} Section 3.5  to efficiently compute rigid meromorphic cocycles. The main difference in this higher weight analogue is that we work with an additive version corresponding to the logarithm of their multiplicative cocycles. 

\begin{proposition} 
    For weight $k>2$, if there exists a lift $\mathcal{J}_{k,\mathcal{D}_1}^{\mathscr{L}}\in H^1(\Gamma,\mathcal{O}_{\mathscr{L}}(2-k,\mathcal{H}_K))$, then there is a $\Gamma$-invariant modular symbol in $MS^\Gamma(\mathcal{O}_{\mathscr{L}}(2-k,\mathcal{H}_K))$ whose associated cocycle is $\mathcal{J}^\mathscr{L}_{k,\mathcal{D}_1}$.
\end{proposition}

\begin{proof}
This follows directly from Proposition \ref{obstructionIso}.
\end{proof}
We will also denote this $\Gamma$-invariant modular symbol by $\mathcal{J}_{k,\mathcal{D}_1}^{\mathscr{L}}.$

Our first aim is to compute $\mathcal{J}_{k,\mathcal{D}_1}^\mathscr{L}\{0,\infty\}$, which by Proposition \ref{modularPeriodProp} determines the modular symbol. Recall the functions $J_{k,\mathcal{D}_1}^{\mathscr{L}}$ defined in Lemma \ref{wt-ncoLem}. Since the functions $J_{k,
\mathcal{D}_1}^\mathscr{L}\{0,\infty\}$ and $\mathcal{J}_{k,
\mathcal{D}_1}^\mathscr{L}\{0,\infty\}$ are both primitives of $J_{k,\mathcal{D}_1}\{0,\infty\}$, the difference $$J_{k,\mathcal{D}_1}^{\mathscr{L}}\{0,\infty\} - \mathcal{J}_{k,\mathcal{D}_1}^{\mathscr{L}}\{0,\infty\} $$ is a polynomial of degree at most $k-2$. We will compute $J_{k,\mathcal{D}_1}^\mathscr{L}\{0,\infty\}$ and use the two-term and three-term relations to adjust it by a polynomial to obtain $\mathcal{J}_{k,\mathcal{D}_1}^\mathscr{L}\{0,\infty\}$. Finally, the function $J_{k,\mathcal{D}_1}^\mathscr{L}\{0,\infty\}$ is determined by its restriction to the standard affinoid of $\mathcal{H}_p$. So we will compute it as a sum of a Mittag--Leffler expansion on the standard affinoid and a finite combination of products of polynomials and logarithms.

By Lemma \ref{expl-nLem} the function we wish to compute is $$ \lim_{j\to \infty}\sum_i\sum_{w\in \Sigma_{\tau_i}^{\leq j}\{0,\infty\}} m_i [(w,\overline{w})\cdot (0, \infty)]\left(\frac{(z-w)^{n/2}(z-\overline{w})^n/2}{(w-\overline{w})^n/2} \log_0(t_w(z)) + p_w(z)\right).$$ 
In theory, one could simply compute this limit. In practice, the size of $\Sigma_{\tau_i}^{\leq j}\{0,\infty\}$ grows exponentially, and quickly becomes infeasible to compute. Recall that $\Phi_{k,\mathcal{D}}$ is the harmonic function associated to $J_{k,\mathcal{D}}$. We will instead define several truncations of $\Phi_{k,\mathcal{D}}$ which we can recursively generate. Define

$$\Phi_{k,\mathcal{D},l}\{r,s\}(e) \defeq \sum_i\sum_{\substack{w\in \Sigma_{\tau_i}\{r,s\}\cap U_e\\\text{level}(w) = l}}m_i [(w,\overline{w})\cdot (r,s)]\frac{(T-w)^{n/2}(T-\overline{w})^{n/2}}{(w-\overline{w})^{n/2}}, $$
and
$$\Phi_{k,\mathcal{D},l}^{(a)}\{r,s\}(e) \defeq \sum_i\sum_{\substack{w\in \Sigma_{\tau_i}\{r,s\}\cap U_e\\\text{level}(w) = l\\w\equiv a\mod p}}m_i [(w,\overline{w})\cdot (r,s)]\frac{(T-w)^{n/2}(T-\overline{w})^{n/2}}{(w-\overline{w})^{n/2}} .$$

Note that $$\Phi_{k,\mathcal{D}} = \sum_{l=0}^\infty \Phi_{k,\mathcal{D},l},$$

 and for $l\geq1$ there is a decomposition $$\Phi_{k,\mathcal{D},l} = \sum_{a\in \mathbb{P}^1(\mathbb{F}_p)}\Phi_{k,\mathcal{D},l}^{(a)}.$$

\begin{definition}
Define the involution $\varpi$ on RM divisors by $$\varpi : [\tau] \rightarrow - [p\tau]$$
\end{definition}

\begin{remark}
The union of the two orbits $[\tau]$ and $[p\tau]$ is the $\mathrm{GL}_2^+\left(\mathbb{Z}[1/p]\right)$-orbit of $\tau$.
\end{remark}

\begin{lemma} \label{measureRecur}
    If $a\in \mathbb{F}_p$, then 
    $$\Phi_{k,-\varpi\mathcal{D},l+1}^{(a)}\{0,\infty\} = \sum_{b\in \mathbb{F}_p}\left(\Phi_{k,\mathcal{D},l}^{(b)}\mid \begin{pmatrix}p & a \\ 0 & 1 \end{pmatrix}\right)\{0,\infty\},$$ 
and     

$$\Phi_{k,-\varpi\mathcal{D},l+1}^{(\infty)}\{0,\infty\} = \left(\Phi_{k,\mathcal{D},l}^{(\infty)}\mid \begin{pmatrix}1 & 0 \\ 0& p \end{pmatrix}\right)\{0,\infty\} +  \sum_{b\in \mathbb{F}_p^\times}\left(\Phi_{k,\mathcal{D},l}^{(b)}\mid \begin{pmatrix}1 & 0 \\ 0 & p \end{pmatrix}\right)\{0,\infty\}.$$
\end{lemma}

\begin{proof}
Observe that an RM-point $w$ is congruent to $a\neq \infty$,  has level $l+1$ and satisfies $$w> 0 > \overline{w}$$ if and only if $\frac{w-a}{p}$ is not congruent to $\infty \mod p$, is of level $l$ and satisfies $$\frac{w-a}{p}> -\frac{a}{p} > \frac{\overline{w}-a}{p}.$$

Similarly, if $w \equiv \infty \mod p$, then $w$ has level $l$ and satisfies $w> 0> \overline{w}$ if and only if $pw$ is not congruent to $0 \mod p$, has level $l+1$ and satisfies $pw > 0 > p\overline{w}$.
\end{proof}

Define the functions associated with these truncations $$F_{k,\mathcal{D},l} \defeq \left\langle \Phi_{k,\mathcal{D},{l+1}},\;  (t-z)^n \log_0(l_z(t)) \right\rangle,$$ 
and
$$F_{k,\mathcal{D},l}^{(a)} \defeq \left\langle \Phi_{k,\mathcal{D},{l+1}}^{(a)},\;  (t-z)^n \log_0(l_z(t)) \right\rangle.$$
Then $$J_{k,\mathcal{D}}^{\mathscr{L}} = \sum_{l=0}^{\infty} F_{k,\mathcal{D},l}$$
and for $l\geq 1$ we have $$ F_{k,\mathcal{D},l} = \sum_{a\in \mathbb{P}^1(\mathbb{F}_p)}F_{k,\mathcal{D},l}^{(a)}$$
The level $0$ function $F_{k,\mathcal{D},0}\{r,s\}$ can be written as $$F_{k,\mathcal{D},0}\{r,s\} = \sum_i \sum_{w\in\Sigma_{\tau_i}\{r,s\} } m_i [(w,\overline{w})\cdot (r,s)]\frac{(z-w)^{n/2}(z-\overline{w})^{n/2}}{(w-\overline{w})^{n/2}}\log_0(z-w) + P_0(z)$$
where $P_0$ is some polynomial of degree at most $n$. 
Then Lemma \ref{measureRecur} implies 

\begin{proposition} \label{recurrenceProp}
    \begin{equation} \label{equation: reccurence 1} F_{k,-\varpi \mathcal{D},l+1}^{(a)}\{0,\infty\} - \sum_{b\in\mathbb{F}_p} \left(F_{k,\mathcal{D},l}^{(b)} \mid  \begin{pmatrix}p & a \\ 0 & 1 \end{pmatrix}\right)\{0,\infty\} \end{equation}

    and 

    \begin{equation} \label{equation: reccurence 2}F_{k,-\varpi \mathcal{D}, l+1}^{(\infty)}\{0,\infty\} - \left(F_{k,\mathcal{D},l}^{(\infty)}\mid \begin{pmatrix}1 & 0 \\ 0 & p \end{pmatrix}\right)\{0,\infty\} - \sum_{b = 1}^{p-1} \left(F_{k,\mathcal{D},l}^{(b)}\mid \begin{pmatrix}1 & 0 \\ 0 & p \end{pmatrix}\right)\{0,\infty\} \end{equation}

    are polynomials of degree at most $n$.

\end{proposition}

Proposition \ref{recurrenceProp} allows us to compute $J_{k,\mathcal{D}}^\mathscr{L}\{0,\infty\} $ modulo polynomials by recursively computing $F_{k,\mathcal{D},l}^{(a)}\{0,\infty\}$. Note that the elements $F_{k,\mathcal{D},l}^{(b)}\{-\frac{a}{p},\infty\}$ can be computed modulo polynomials by finding a unimodular path between the cusps $-\frac{a}{p}$ to $\infty$, and then using the $\mathrm{SL}_2(\mathbb{Z})$-invariance of $\{F_{k,\mathcal{D},l}^{(a)}\}_a$.

Once we have computed $J_{k,\mathcal{D}_1}^\mathscr{L}$ modulo the addition of a polynomial we can use the two term relation,
 $$ \mathcal{J}^\mathscr{L}_{k,\mathcal{D}}\{0,\infty\} |S + \mathcal{J}^{\mathscr{L}}_{k,\mathcal{D}}\{0,\infty\} = 0,$$  and the three term relation, $$\mathcal{J}^{\mathscr{L}}_{k,\mathcal{D}}\{0,\infty\} |U^2+\mathcal{J}^{\mathscr{L}}_{k,\mathcal{D}}\{0,\infty\} |U + \mathcal{J}^{\mathscr{L}}_{k,\mathcal{D}}\{0,\infty\} = 0,$$ to find a polynomial which fixes the $\Gamma$-invariance. This amounts to solving a set of simultaneous linear equations.

\begin{example}\label{polyfixEx}
When $k = 4$ and $n=2$. We have computed a function $$F(z) = \mathcal{J}^\mathscr{L}_{4,\mathcal{D}}\{0,\infty\}(z) + P(z),$$ for some quadratic polynomial $P = az^2 + bz + c$.  Now applying the two-term and three-term relations we obtain $$F +  (F\vert S)  = P + (P \vert S) = (a+c)(z^2+1),$$ and $$F + (F\vert U)  + (F\vert U^2) = P +  (P \vert U) + (P \vert U^2)  = (2a+b+2c)(z^2-z+1).$$
This determines the polynomial $P$ up to multiples of $z^2 - 1$.

But we note that the coboundary in $P_n$ given by $$\gamma \mapsto (1\mid \gamma) - 1$$ sends $S$ to $z^2-1$, therefore the ambiguity of a multiple of $z^2-1$ is a coboundary and does not affect the evaluation.
\end{example}

\begin{algorithm}
\label{algorithm}
\begin{itemize} 
    \mbox{}
\setlength\itemsep{0.5em}
\item \textbf{Step 1.} By computing reduced binary quadratic forms in $\mathrm{SL}_2(\mathbb{Z})$-equivalence classes, compute $\Phi_{k,\mathcal{D}_1,0}\{0,\infty\}$ and store $F_{k,\mathcal{D}_1,0}\{0,\infty\}$ as a tuple of pairs $(w,P)$ of points on $\mathcal{H}_p$ and polynomials which record the data of functions that are sums of $$P(z)\log_{\mathscr{L}}(z-w),$$ where $P$ is some polynomial of degree at most $k-2$, and $w$ an RM-point.

\item \textbf{Step 2.} Similarly to the previous step, compute the $(p+1)$ functions $$F_{k,\mathcal{D}_1,1}^{(a)}\{0,\infty\}$$ and store them to sufficient precision as a Mittag--Leffler expansion on the standard affinoid.
\item \textbf{Step 3.} Use the recursion equations \emph{(\ref{equation: reccurence 1})}and \emph{(\ref{equation: reccurence 2})} to compute $$F_{k,\mathcal{D}_1,l}^{(a)}$$ modulo addition of a polynomial to a sufficiently high level, and sum the resulting functions. 
\item \textbf{Step 4.} Use the two and three-term relations to fix the polynomial part and obtain $$\mathcal{J}_{k,\mathcal{D}_1}^{\mathscr{L}}\{0,\infty\}$$ as in Example \ref{polyfixEx}.
\item \textbf{Step 5.} The function $$\mathcal{J}^{\mathscr{L}}_{k,\mathcal{D}_1}(\gamma_\sigma)$$ is equal to $$\mathcal{J}_{k,\mathcal{D}_1}^{\mathscr{L}}\{ 0,\gamma_\sigma \cdot 0\},$$ and can be computed by finding a unimodular sequence $$0 = s_1, \; s_2, \; \cdots ,\;  s_{r-1},\; s_r = \sigma \cdot 0$$ and matrices $\gamma_i \in \mathrm{SL}_2(\mathbb{Z})$ such that $\gamma_i\cdot \{ 0,\infty \} = \{s_i, s_{i+1}\}$ and applying the $\Gamma$-invariance of the modular symbol $\mathcal{J}^{\mathscr{L}}_{k,\mathcal{D}_1}$. $$\mathcal{J}_{k,\mathcal{D}_1}^{\mathscr{L}}(\gamma_\sigma) = \sum_i \mathcal{J}_{k,\mathcal{D}_1}^{\mathscr{L}}\{0,\infty\} \mid {\gamma_i}$$
\item \textbf{Step 6.} The value of $$\mathcal{J}^{\mathscr{L}}_{k,\mathcal{D}_1}[y_{\mathcal{D}_2}]$$ is then given by the sum \begin{equation}\label{eqn-pairing}
\sum_{j=0}^{n/2} \frac{(-1)^i \binom{n-i}{m-i}}{i!\binom{n}{m}}\frac{\left(\mathcal{J}^{\mathscr{L}}_{k,\mathcal{D}_1}(\gamma_\sigma)\right)^{(j)}(\sigma)}{(\sigma - \overline{\sigma})^{n/2-j}}.
\end{equation}
\end{itemize}
\end{algorithm}

\begin{remark}
Note that the expression in (\ref{eqn-pairing}) is equal to formally applying iterations of the raising operator $\delta_{-i}$ and then evaluating at $\sigma$: $$\delta_{2-k}^{(k-2)/2}\left(\mathcal{J}^{\mathscr{L}}_{k,\mathcal{D}_1}(\gamma_i)\right)(\sigma),$$ where $$\delta_i = \frac{\mathrm{d}}{\mathrm{d}z} + \frac{i}{z-\overline{z}}.$$
\end{remark}
\subsection{Results and algebraicity conjectures}
We have implemented Algorithm \ref{algorithm} in the computer algebra system SageMath. The program and instructions on using it are available on the author's website. 

If $p = 3$ and $k=4$, then the space of newforms $S_4(\Gamma_0(3))^{p-new}$ is trivial. This implies that for every RM-divisor $\mathcal{D}$, there is a unique lift of $J_{k,\mathcal{D}}$ to $$\mathcal{J}^{\mathscr{L}}_{k,\mathcal{D}} \in \mathrm{MS}^\Gamma(\mathcal{O}_{\mathcal{L}}(-2,\mathcal{H}_K)), $$ and also $$\mathcal{J}_{k,\mathcal{D}}^o \in \mathrm{MS}^\Gamma(C^0_{K,har}(P_n)).$$

\begin{example}
Let $$\mathcal{D}_1 = \left(9[\phi] - 2 [\sqrt{5}]\right) + \varpi \left(9[\phi] - 2 [\sqrt{5}]\right).$$
Then we can verify that $\mathrm{Deg}_{4,\mathcal{D}_1}$ is trivial, and therefore we can use \ref{algorithm} to compute $$\mathcal{J}_{4,\mathcal{D}_1}^\mathscr{L}$$ and its values. 

We find that up to 145 digits of $3$-adic accuracy 
\begin{equation}
\mathcal{J}_{4,\mathcal{D}_1}^{0}[4\phi] = \frac{1}{40}\log_0\left(\left(\frac{i+2}{i-2}\right)^5\left(\frac{i+5}{i-5} \right)^{-26}\left(\frac{i+4}{i-4}\right)^{-63}\left(\frac{i+6}{i-6}\right)^{-17} \right),
\end{equation}

and 

\begin{align}
    \begin{split}
    \mathcal{J}_{4,\mathcal{D}_1}^{0}[7\phi] = \frac{1}{140}\log_0\left(\left(\frac{\sqrt{-7}+1}{\sqrt{-7}-1} \right)^{-2241} \left(\frac{\sqrt{-7}+4}{\sqrt{-7}-4} \right)^{-167}\right. \\ \left. \times \left(\frac{\sqrt{-7}+17}{\sqrt{-7}-17}\right)^8 \left(\frac{\sqrt{-7}+6}{\sqrt{-7}-6} \right)^{16}\left(\frac{\sqrt{-7}+23}{\sqrt{-7}-23} \right)^{-64} \right).
    \end{split}
\end{align}
\end{example}

\begin{example}
    Let $$\mathcal{D}_1 = \left(7 [\sqrt{2}] - 4[2\sqrt{2}] \right) + \varpi\left(7[\sqrt{2}] - 4[2\sqrt{2}]  \right).$$ Again, $\mathrm{Deg}_{4, \mathcal{D}_1}$ is trivial. Up to 145 digits of $3$-adic accuracy we have 
    \begin{equation}
        \mathcal{J}_{4,\mathcal{D}_1}^0[4\sqrt{2}] = \frac{1}{64}\log_0\left(\left(\frac{i +2 }{i-2} \right)^{-208} \left(\frac{i+5}{i-5} \right)^{-30} \left( \frac{i+3}{i-3}\right)^{-91} \right).
    \end{equation}
    \begin{equation}
        \mathcal{J}_{4,\mathcal{D}_1}^0[\phi] = \frac{1}{\sqrt{160}}\log_0\left( \left(\frac{\sqrt{-10}+ 4}{\sqrt{-10}-4} \right)^2 \right).
    \end{equation}
\end{example}

\begin{conjecture}\label{algConj}
Let $$\mathcal{D}_1 = \sum_i m_i [\tau_i]$$ be a principal RM-divisor, and $\sigma$ an RM-point. Denote by $d_{1i}$ the discriminant of $\tau_i$ and $d_2$ the discriminant of $\sigma$. Then there exist algebraic numbers $\alpha_i$ such that for every $\mathscr{L}\in \mathbb{C}_p$ we have
$$ \mathcal{J}_{k,\mathcal{D}_1}^{\mathscr{L}}[\sigma] = \sum_i |d_{1i}d_2|^{(2-k)/4}\log_\mathscr{L}(\alpha_i),$$ and $$\mathcal{J}^o_{k,\mathcal{D}_1} = \sum_i |d_{1i}d_2|^{(2-k)/4}\mathrm{ord}_p(\alpha_i),$$ where $\alpha_i$ lies in the compositum $H_{\tau_i} \cdot  H_\sigma$.
\end{conjecture}

\subsection{Prime factorization}

Keeping the same notation as in Conjecture \ref{algConj}, we will now give a conjectural prime factorization for the algebraic numbers appearing in Conjecture \ref{algConj}. In particular, for a prime $q\neq p$ a prime number we will give a conjectural value for $$ |d_{1i}d_2|^{(2-k)/4}\mathrm{ord}_q(\alpha_i).$$

Let $\mathcal{O}$ be an order of discriminant $D$ in a real-quadratic field where both $p$ and $q$ are not split. Let $B/\mathbb{Q}$ be the quaternion algebra ramified only at $p$ and $q$, and let $\mathcal{R}$ be a choice of maximal order in $B$. Then there is a bijection $$ \Gamma \backslash \mathcal{H}_p^D \cong \Sigma(\mathcal{O},\mathcal{R}),$$ where $\Sigma(\mathcal{O},\mathcal{R})$ is the set of $\mathcal{R}_1^\times$ conjugacy classes of oriented optimal embeddings of $\mathcal{O}$ into $\mathcal{R}$. Note that such a bijection is not canonical. We will fix bijections which are compatible with the action of the class group. See \cite{DV1} Section 3.4 for the definitions and more details.

\begin{definition}
    For any two distinct optimal embeddings of quadratic orders $$\varphi_i : \mathcal{O}_i \rightarrow \mathcal{R},\; i\in\{1,2\},$$ we define $$[\varphi_1, \varphi_2]_{q}\defeq \mathrm{max }\; t\geq 1\;  \vert \; \varphi_1(\mathcal{O}_1) , \varphi_2(\mathcal{O}_2) \text{ have the same image in } \mathcal{R}/q^{t-1}\mathcal{R}.$$
\end{definition}

\begin{conjecture} \label{conj: factorization}
    Keeping the same notation as Conjecture \ref{algConj}, let $\varphi_{1i}$ be embeddings of the orders $\mathcal{O}_{\tau_i}$ associated to $[\tau_i]$ and let $\varphi_2$ be an embedding of $\mathcal{O}_{\sigma}$ associated to $\sigma$. Then for a prime $q \neq p$ which is not split in either order $\mathcal{O}_{\tau_i}$ or $\mathcal{O}_{\sigma}$ and does not divide any of $d_{1i}$ or $d_2$, there is an embedding of $H_{\tau_i}\cdot H_\sigma$ in $\overline{\mathbb{Q}_q}$ such that 
    $$|d_{1i}d_2|^{(2-k)/4} \mathrm{ord}_q(\alpha_i) = \sum_{\gamma \in\gamma_2^\mathbb{Z} \backslash  \mathcal{R}_1^\times / \gamma_1^\mathbb{Z}} [\gamma\varphi_{1i} \gamma^{-1}, \varphi_2]_q\cdot [(\gamma \tau_i, \overline{\gamma \tau_i}) \cdot (\sigma, \overline{\sigma})] \cdot \left\langle P_{\gamma \tau_1}, P_{\tau_2} \right\rangle _{P_n}.$$ Otherwise, if $q$ is split in one of the two orders then $$\mathrm{ord}_q(\alpha_i) = 0.$$
\end{conjecture}

\bibliographystyle{alpha} 
\bibliography{mybib}{}

\begin{thebibliography}{MTT86}

\bibitem[AS86]{AshStevens}
Avner Ash and Glenn Stevens.
\newblock Modular forms in characteristic {$l$} and special values of their {$L$}-functions.
\newblock {\em Duke Math. J.}, 53(3):849--868, 1986.

\bibitem[BD09]{BD}
Massimo Bertolini and Henri Darmon.
\newblock The rationality of {S}tark-{H}eegner points over genus fields of real quadratic fields.
\newblock {\em Ann. of Math. (2)}, 170(1):343--370, 2009.

\bibitem[BEY21]{BEY}
Jan~Hendrik Bruinier, Stephan Ehlen, and Tonghai Yang.
\newblock C{M} values of higher automorphic {G}reen functions for orthogonal groups.
\newblock {\em Invent. Math.}, 225(3):693--785, 2021.

\bibitem[BLY22]{BLY}
Jan~H. Bruinier, Yingkun Li, and Tonghai Yang.
\newblock Deformations of theta integrals and a conjecture of {G}ross-{Z}agier, 2022.

\bibitem[Bre04]{Breuil1}
Christophe Breuil.
\newblock Invariant {$ L$} et s\'erie sp\'eciale {$p$}-adique.
\newblock {\em Ann. Sci. \'Ecole Norm. Sup. (4)}, 37(4):559--610, 2004.

\bibitem[Bre10]{Breuil2}
Christophe Breuil.
\newblock S\'erie sp\'eciale {$p$}-adique et cohomologie \'etale compl\'et\'ee.
\newblock {\em Ast\'erisque}, (331):65--115, 2010.

\bibitem[Col82]{Coleman}
Robert~F. Coleman.
\newblock Dilogarithms, regulators and {$p$}-adic {$L$}-functions.
\newblock {\em Invent. Math.}, 69(2):171--208, 1982.

\bibitem[CZ93]{CZ}
Yj. Choie and D.~Zagier.
\newblock Rational period functions for {${\rm PSL}(2,\bold Z)$}.
\newblock In {\em A tribute to {E}mil {G}rosswald: number theory and related analysis}, volume 143 of {\em Contemp. Math.}, pages 89--108. Amer. Math. Soc., Providence, RI, 1993.

\bibitem[Dar01]{Darmon}
Henri Darmon.
\newblock Integration on $h_p\times h$ and applications.
\newblock {\em Annals of Mathematics 154}, 2001.

\bibitem[dS09]{dSBounded}
Ehud de~Shalit.
\newblock Bounded cohomology of the {$p$}-adic upper half plane.
\newblock In {\em Symmetries in algebra and number theory ({SANT})}, pages 27--47. Universit\"atsverlag G\"ottingen, G\"ottingen, 2009.

\bibitem[DT08]{DT}
S.~Dasgupta and J.~Teitelbaum.
\newblock {\em The {$p$}-adic upper half plane, {$p$}-adic geometry}.
\newblock Univ. Lecture Ser, 45 edition, 2008.

\bibitem[DV21]{DV1}
Henri Darmon and Jan Vonk.
\newblock Singular moduli for real quadratic fields: a rigid analytic approach.
\newblock {\em Duke Math Journal, 170, Number 1}, 2021.

\bibitem[DV22]{DV3}
Henri Darmon and Jan Vonk.
\newblock Arithmetic intersections of modular geodesics.
\newblock {\em J. Number Theory}, 230:89--111, 2022.

\bibitem[GKZ87]{GKZ}
B.~Gross, W.~Kohnen, and D.~Zagier.
\newblock {\em Math. Ann.}, 278:497 -- 562, 1987.

\bibitem[GS94]{GS}
H.~Gillet and C.~Soul\'e.
\newblock Arithmetic analogs of the standard conjectures.
\newblock In {\em Motives ({S}eattle, {WA}, 1991)}, volume 55, Part 1 of {\em Proc. Sympos. Pure Math.}, pages 129--140. Amer. Math. Soc., Providence, RI, 1994.

\bibitem[GZ84]{GZ}
B.~H. Gross and D.~B. Zagier.
\newblock Heegner points and derivatives of $l$-series.
\newblock {\em Inventiones Math}, 1984.

\bibitem[Hid93]{Hida}
Haruzo Hida.
\newblock {\em Elementary theory of {$L$}-functions and {E}isenstein series}, volume~26 of {\em London Mathematical Society Student Texts}.
\newblock Cambridge University Press, Cambridge, 1993.

\bibitem[Li22]{Li}
Yingkun Li.
\newblock Average {CM}-values of higher {G}reen's function and factorization.
\newblock {\em Amer. J. Math.}, 144(5):1241--1298, 2022.

\bibitem[Mel09]{mellit}
Anton Mellit.
\newblock {\em Higher Green’s functions for modular forms}.
\newblock PhD thesis, Rheinische Friedrich-Wilhelms-Universität Bonn, 2009.

\bibitem[Mor84]{morita}
Y~Morita.
\newblock Analytic representations of $\textrm{SL}_2$ over a $\mathfrak{p}$-aidc number field. ii.
\newblock In {\em Automorphic forms of several variables (Katata, 1983)}. Birkh\"{a}user, 1984.

\bibitem[MTT86]{MTT}
B.~Mazur, J.~Tate, and J.~Teitelbaum.
\newblock On $p$-adic analogues of the conjectures of birch and swinnerton-dyer.
\newblock {\em Inventiones mathematicae}, 84, 1986.

\bibitem[Neg22]{NegThesis}
Isabella Negrini.
\newblock {\em A Shimura-Shintani correspondence for rigid analytic cocycles of higher weight}.
\newblock PhD thesis, McGill University, 2022.

\bibitem[Neg23]{Neg1}
Isabella Negrini.
\newblock A {S}himura-{S}hintani correspondence for rigid analytic cocycles of higher weight.
\newblock {\em Forum Math.}, 35(2):549--571, 2023.

\bibitem[Ort04]{Orton}
Louisa Orton.
\newblock An elementary proof of a weak exceptional zero conjecture.
\newblock {\em Canad. J. Math.}, 56(2):373--405, 2004.

\bibitem[RS12]{RS}
Victor Rotger and Marco~Adamo Seveso.
\newblock {$ \mathcal{L}$}-invariants and {D}armon cycles attached to modular forms.
\newblock {\em J. Eur. Math. Soc. (JEMS)}, 14(6):1955--1999, 2012.

\bibitem[Sev12]{Seveso}
Marco~Adamo Seveso.
\newblock {$p$}-adic {$L$}-functions and the rationality of {D}armon cycles.
\newblock {\em Canad. J. Math.}, 64(5):1122--1181, 2012.

\bibitem[Via15]{Viazovska}
Maryna Viazovska.
\newblock C{M} values of higher {G}reen's functions and regularized {P}etersson products.
\newblock In {\em Arithmetic and geometry}, volume 420 of {\em London Math. Soc. Lecture Note Ser.}, pages 493--503. Cambridge Univ. Press, Cambridge, 2015.

\bibitem[Zha97]{zhang}
Shouwu Zhang.
\newblock Heights of {H}eegner cycles and derivatives of {$L$}-series.
\newblock {\em Invent. Math.}, 130(1):99--152, 1997.

\end{thebibliography}
\end{document}